\documentclass[11pt]{amsart}

\usepackage[mathscr]{eucal}
\usepackage{amsmath,amssymb,amsfonts,amsthm,enumerate}

\newtheorem{theorem}{Theorem}[section]
\newtheorem{lemma}[theorem]{Lemma}

\newtheorem{theirtheorem}{Theorem}

\newcommand{\Z}{\mathbb Z}

\DeclareMathOperator{\ord}{ord}

\DeclareMathOperator{\supp}{Supp}

\newcommand{\la}{\langle}
\newcommand{\ra}{\rangle}
\newcommand{\be}{\begin{equation}}
\newcommand{\ee}{\end{equation}}
\newcommand{\und}{\;\mbox{ and }\;}
\newcommand{\nn}{\nonumber}
\newcommand{\ber}{\begin{eqnarray}}
\newcommand{\eer}{\end{eqnarray}}
\newcommand{\Sum}[2]{\underset{#1}{\overset{#2}{\sum}}}
\newcommand{\Summ}[1]{\underset{#1}{\sum}}

\newcommand{\Fc}{\mathcal F}
\newcommand{\vp}{\mathsf v}

\DeclareSymbolFont{goo}{OMS}{cmsy}{b}{n}
\DeclareMathSymbol{\gooT}{\mathalpha}{goo}{"1}
\newcommand{\bdot}{\mathbin{\gooT}}
\begin{document}

\title{Representing Sequence Subsums as Sumsets of Near Equal Sized Sets}
\author{David J. Grynkiewicz}

\begin{abstract}
For a sequence $S$ of terms from an abelian group $G$ of length $|S|$, let $\Sigma_n(S)$ denote the set of all elements that can be represented as the sum of terms in some $n$-term subsequence of $S$. When the subsum set is very small, $|\Sigma_n(S)|\leq |S|-n+1$, it is known that the terms of $S$ can be partitioned into $n$ nonempty sets $A_1,\ldots,A_n\subseteq G$ such that $\Sigma_n(S)=A_1+\ldots+A_n$. Moreover, if the upper bound is strict, then $|A_i\setminus Z|\leq 1$ for all $i$, where $Z=\bigcap_{i=1}^{n}(A_i+H)$ and $H=\{g\in G:\; g+\Sigma_n(S)=\Sigma_n(S)\}$ is the stabilizer of $\Sigma_n(S)$. This allows structural results for sumsets to be used to study the subsum set $\Sigma_n(S)$ and is one of the two main ways to derive the natural subsum analog of Kneser's Theorem for sumsets. In this paper, we show that such a partitioning can be achieved with sets $A_i$ of as near equal a size as possible, so $\lfloor \frac{|S|}{n}\rfloor \leq |A_i|\leq \lceil\frac{|S|}{n}\rceil$ for all $i$, apart from one highly structured counterexample when $|\Sigma_n(S)|= |S|-n+1$ with $n=2$. The added information of knowing the sets $A_i$ are of near equal size can be of use when applying the aforementioned partitioning result, or when applying sumset results to study $\Sigma_n(S)$ (e.g., \cite{Alfred-Dihedral}).
 We also give an extension increasing the flexibility of the aforementioned partitioning result and prove some stronger results when $n\geq \frac12|S|$ is very large.
\end{abstract}

\maketitle

\section{Introduction}

\subsection*{Basic Notation} Let $G$ be an abelian group. Following standard conventions in Combinatorial Number Theory (see \cite{Gbook} \cite{alfredbook} \cite{Alfred-Ruzsa-book}), by a sequence $S$ of terms from $G$, we mean a finite, \emph{unordered} string of elements $$S=g_1\bdot\ldots\bdot g_\ell$$ with $g_i\in G$  the terms of the sequence $S$, each term separated via the boldsymbol $\bdot$ (differentiating it from multiplication in circumstances where both operations are in use). Formally, a sequence is considered as an element of the free abelian monoid $\mathcal F(G)$ with basis $G$ and operation $\bdot$, giving a standardized system of notation for sequences. Given an element $g\in G$, we let $\vp_g(S)\geq 0$ denote the number of occurrences of the term $g$ in $S$ and let $g^{[n]}$ represent the sequence consisting of the element $g$ repeated $n$ times, so that any sequence $S\in\Fc(G)$ has the form $$S=\underset{g\in G}{{\prod}}^\bullet g^{[\vp_g(S)]}.$$
We let $T\mid S$ denote that $T$ is a subsequence of $S$, so $\vp_g(T)\leq \vp_g(S)$ for all $g\in G$, and in such case use $T^{[-1]}\bdot S$ or $S\bdot T^{[-1]}$ to denote the sequence obtained by removing from $S$ the terms in $T$, so $\vp_g(T^{[-1]}\bdot S)=\vp_g(S)-\vp_g(T)$.
The support of the sequence $S$ is the set of all elements occurring in $S$:
$$\supp(S)=\{g\in G:\;\vp_g(S)>0\}.$$
For a subset $X\subseteq G$, let $S_X\mid S$ denote the subsequence of $S$ consisting of all terms from $X$, so $$S_X={\prod_{g\in X}}^\bullet g^{[\vp_g(S)]}.$$
Then $|S|=\ell$ is the length of the sequence, $$\mathsf h(S)=\max\{\vp_g(S):\; g\in G\}$$ is the maximum multiplicity of a term in $S$, $\sigma(S)=g_1+\ldots+g_n$ is the sum of $S$, and $$\Sigma_n(S)=\{\sigma(T):\;T\mid S,\;|T|=n\}\subseteq G$$ is the set of $n$-term subsums of $S$, for $n\geq 0$.

All intervals are  discrete, so $[m,n]=\{x\in \Z:\; m\leq x\leq n\}$. Given subsets $A_1,\ldots,A_n\subseteq G$, their sumset is defined as $$A_1+\ldots+A_n=\{a_1+\ldots+a_n:\:a_i\in A_i\}.$$ The stabilizer of a set $A\subseteq G$ is the subgroup $\mathsf H(A)=\{g\in G:\;g+A=A\}\leq G$, which is the largest subgroup $H$ such that $A$ is a union of $H$-cosets. If $\mathsf H(A)$ is trivial, then $A$ is \emph{aperiodic}, and otherwise $A$ is \emph{periodic}. We say that $A$ is \emph{$H$-periodic} if $A$ is a union of $H$-cosets, equivalently, if $H\leq \mathsf H(A)$. For $x\in G$ and $A,\,B\subseteq G$, we let $\mathsf r_{A+B}(x)=|(A-x)\cap B|=|\{(a,b)\in A\times B:\; a+b=x\}|$ denote the number of representations for $x$ as an element of $A+B$, and call $x\in A+B$ a unique expression element  when $\mathsf r_{A+B}(x)=1$. For $H\leq G$, we let $\phi_H:G\rightarrow G/H$ denote the natural homomorphism.

\subsection*{Background} The study of sequence subsums is a classical topic in Combinatorial Number Theory. Often, it is desired that either $0\in \Sigma_n(A)$, or $|\Sigma_n(S)|$ is large, or $\Sigma_n(S)=G$, and either conditions that guarantee the appropriate outcome, or the structure of sequences failing to satisfy the desired outcome, are sought. The Erd\H{o}s-Ginzburg-Ziv Theorem \cite{Gbook} \cite{natboook} \cite{egz} and Davenport Constant \cite{Gbook} \cite{alfredbook} \cite{IntegerRamseyBook} are two such examples of very well-studied problems along these lines. A selection of other examples may be found here \cite{extra5}  \cite{extra4} \cite{extra1} \cite{wolfgang-ordaz-olson-constant} \cite{extra3} \cite{extra2}.

One effective tool for studying $\Sigma_n(S)$, e.g., employed in the original proof of the Erd\H{o}s-Ginzburg-Ziv Theorem \cite{egz}, is via setpartitions. Consider a sequence $\mathscr A=A_1\bdot \ldots\bdot A_n$ whose terms $A_i$ are \emph{nonempty} (and finite)  subsets of $G$. We call such a sequence a \emph{setpartition} over $G$. Note the  setpartition $\mathscr A$ naturally partitions the terms in its underlying sequence $$\mathsf S(\mathscr A):={\prod_{i\in [1,n]}}^\bullet{\prod_{g\in A_i}}^\bullet g$$ into $n$ nonempty sets. It is then rather immediate that $\Sum{i=1}{n}A_i\subseteq \Sigma_n(S)$ when $\mathsf S(\mathscr A)\mid S$, which allows sumset results to be used for studying $\Sigma_n(S)$. This becomes even more effective if we know there is some setpartition $\mathscr A=A_1\bdot\ldots\bdot A_n$ with $\mathsf S(\mathscr A)\mid S$ such that equality holds, $\Sum{i=1}{n}A_i=\Sigma_n(S)$, for this means the subsums $\Sigma_n(A)$ can be represented as an ordinary sumset, and sumset results directly applied. The more structure that is known for the $A_i$, the easier and more effective it is to apply the corresponding sumset results. While this cannot hold for a general sequence, we have the striking fact that this is always possible so long as $|\Sigma_n(S)|$ is sufficiently small \cite[Theorem 14.1]{Gbook} \cite{ccd}.

\begin{theirtheorem}[Partition Theorem]\label{thm-ccd-lite}
Let $G$ be an abelian group, let $n\geq 1$, let $S\in \mathcal F(G)$ be a sequence of terms from $G$, and suppose $S'\mid S$ is a subsequence  with $\mathsf h(S')\leq n\leq |S'|$. Then there exists a setpartition $\mathscr A=A_1\bdot\ldots\bdot A_n$ with $\mathsf S(\mathscr A)\mid S$ and $|\mathsf S(\mathscr A)|=|S'|$ such that either
\begin{itemize}
\item[1.]$|\Sigma_n(S)|\geq |\Sum{i=1}{n}A_i|\geq |S'|-n+1$, or
\item[2.] $\Sigma_n(S)=\Sum{i=1}{n}A_i$, \ $\supp(\mathsf S(\mathscr A)^{[-1]}\bdot S)\subseteq Z$ and $|A_i\setminus Z|\leq 1$ for all $i$, where $Z=\bigcap_{i=1}^{n}(A_i+H)$ and $H=\mathsf H(\Sigma_n(S))$.
\end{itemize}
\end{theirtheorem}

Theorem \ref{thm-ccd-lite} ensures that $\Sigma_n(S)=\Sum{i=1}{n}A_i$, for some setpartition $\mathscr A=A_1\bdot\ldots\bdot A_n$ with $\mathsf S(\mathscr A)\mid S$ and $|\mathsf S(\mathscr A)|=|S'|$, provided $|\Sigma_n(S)|\leq |S'|-n+1$, with additional structural information holding when the upper bound is strict. Worth noting, Theorem \ref{thm-ccd-lite} can always be applied (so long as $|S|\geq n$) with $S'$ taken to be the maximal subsequence of $S$ with $\mathsf h(S')\leq n$.  In case Theorem \ref{thm-ccd-lite}.2 holds, this allows us to apply Kneser's Theorem \cite{kneserstheorem} \cite{Gbook} \cite{natboook} \cite{alfredbook} to derive yet more information regarding $\Sigma_n(S)$, which is often incorporated into the statement of Theorem \ref{thm-ccd-lite} itself (e.g. \cite[Theorem 14.1]{Gbook} \cite{IttII}).

\begin{theirtheorem}[Kneser's Theorem] Let $G$ be an abelian group, let $A_1,\ldots,A_n\subseteq G$ be finite, nonempty subsets, and let $H=\mathsf H(\Sum{i=1}{n}A_i)$. Then
$$|\Sum{i=1}{n}A_i|\geq \Big(\Sum{i=1}{n}|\phi_H(A_i)|-n+1\Big)|H|=\Sum{i=1}{n}|A_i+H|-(n-1)|H|.$$
\end{theirtheorem}

Kneser's Theorem is the fundamental lower bound for sumsets in an abelian group. Combining it with Theorem \ref{thm-ccd-lite} (applied modulo $H$) yields the analogous result for sequence subsums \cite{IttII}.

\begin{theirtheorem}[Subsum Kneser's Theorem]
\label{thm-subsum-kneser}
Let $G$ be an abelian group, let $n\geq 1$, let $S\in \mathcal F(G)$ be a sequence of terms from $G$ with $|S|\geq n$, and let $H=\mathsf H(\Sigma_n(S))$.
Then $$|\Sigma_n(S)|\geq \Big(|\phi_H(S')|-n+1\Big)|H|,$$
where $S'\mid S$ is a maximum length subsequence with $\mathsf h(\phi_H(S'))\leq n$.
\end{theirtheorem}

Note $|\phi_H(S')|=\Summ{g\in G/H}\max\{n,\,\vp_g(\phi_H(S))\}$. The Subsum Kneser's Theorem can alternatively be derived as  a special case of the DeVos-Goddyn-Mohar Theorem \cite{DGM} \cite{Gbook}. Theorem \ref{thm-subsum-kneser}, and the more general Theorem \ref{thm-ccd-lite}, have found numerous use in problems regarding sequence subsums \cite{arie-II}  \cite{PropB} \cite{Gao-nsum-paper} \cite{Gao-bonus} \cite{GG-nonabelian-index2} \cite{nondescrea-diam-zerosum} \cite{hamconj} \cite{hypergraph-egz} \cite{wegz} \cite{number-zs} \cite{GrahamConj} \cite{rasheed} \cite{andy-paper} \cite{oscar1} \cite{oscar-gaothms}, extending, complementing or resolving questions of established interest \cite{bialostocki-I} \cite{Bialostocki-II} \cite{bialostock-nondecreasing} \cite{bollobas-leader} \cite{Cao} \cite{Caro-WEGZ-survey}  \cite{egz}  \cite{graham-original} \cite{grahamconj} \cite{Furedi-kleitman-m-term-zs} \cite{Gao-preolson} \cite{Gao-conj-ordaz-paper} \cite{ham-ordaz} \cite{ham-subsumConj} \cite{kisin} \cite{Mann-preolson} \cite{olson1-pregao} \cite{pingzhui-zeng}.

In this paper, we will further strengthen Theorem \ref{thm-ccd-lite}. Theorem \ref{thm-main-ccd} applies to the more general object $X+\Sigma_n(S)$ rather than $\Sigma_n(S)$ (which is the case $X=\{0\}$), showing that Theorem \ref{thm-ccd-lite} holds even if a fixed portion is ``frozen'' in the set $X$. For instance, if $X\subseteq \Sigma_m(T)$, then $X+\Sigma_n(S)\subseteq \Sigma_{m+n}(T\bdot S)$, and we obtain the conclusion of Theorem \ref{thm-ccd-lite} under the restriction of only being able to repartition the terms from $S$. Theorem \ref{thm-main-ccd} also shows that, apart from one highly structured counter-example characterized in Theorem \ref{thm-main-ccd}.3, the resulting setpartition $\mathscr A=A_1\bdot\ldots\bdot A_n$ can be chosen such that the sizes of the sets $A_i$ are as near equal as possible, i.e., with $||A_i|-|A_j||\leq 1$ for all $i,\,j\in [1,n]$ (equivalently,
 $\lfloor \frac{|\mathsf S(\mathscr A)|}{n}\rfloor \leq |A_i|\leq \lceil\frac{|\mathsf S(\mathscr A)|}{n}\rceil$ for all $i$). We call such a setpartition \emph{equitable}. While such improvements are not needed for every application of Theorem \ref{thm-ccd-lite}, they can simplify technical issues related to the use of Theorem \ref{thm-ccd-lite}, sometimes in an  essential fashion. For example, the results of this paper (Sections \ref{sec-ngen} and \ref{sec-nlarge}) are needed to prove the  main result in the forthcoming paper \cite{Alfred-Dihedral} dealing with refined properties of product-one sequences over a dihedral group.

\begin{theorem}\label{thm-main-ccd}
Let $G$ be an abelian group, let $n\geq 1$, let $X\subseteq G$ be a finite, nonempty set, let $L\leq \mathsf H(X)$, let $S\in \mathcal F(G)$ be a sequence of terms from $G$, and suppose $S'\mid S$ is a subsequence  with $\mathsf h(\phi_L(S'))\leq n\leq |S'|$. Then there is a setpartition $\mathscr A=A_1\bdot\ldots\bdot A_n$ with $\mathsf S(\mathscr A)\mid S$, $|\mathsf S(\mathscr A)|=|S'|$ and $|\phi_L(A_i)|=|A_i|$ for all $i\in [1,n]$ such that
\begin{itemize}
\item[1.]$|X+\Sigma_n(S)|\geq |X+\Sum{i=1}{n}A_i|\geq (|S'|-n)|L|+|X|$ and $\mathscr A$ is equitable, or
\item[2.] $X+\Sigma_n(S)=X+\Sum{i=1}{n}A_i$, \ $\mathscr A$ is equitable,  $\supp(\mathsf S(\mathscr A)^{[-1]}\bdot S)\subseteq Z$ and $|A_i\setminus Z|\leq 1$ for all $i$, where $Z=\bigcap_{i=1}^{n}(A_i+H)$ and $H=\mathsf H(X+\Sigma_n(S))$, or
\item[3.] $n=2$, $X\setminus (\beta+L)$ and $(A_1+L)\cap (A_2+L)$ are $K$-periodic, $\supp(\mathsf S(\mathscr A)^{[-1]}\bdot S)\subseteq (A_1+L)\cap (A_2+L)$, \ $\big((A_1+L)\cup (A_2+L)\big)\setminus \big((A_1+L)\cap (A_2+L)\big)$ is a $K$-coset,  \  $\mathsf H(X+\Sigma_n(S))=\mathsf H(X+\Sum{i=1}{n}A_i)=\mathsf H(X)=L$, and $|X+\Sigma_n(S)|=(|S'|-n)|L|+|X|$,  for some $\beta\in X$ and  $K\leq G$ with $L\leq K$ and  $K/L\cong (\Z/2\Z)^2$.
\end{itemize}
\end{theorem}

In Section \ref{sec-nlarge}, we will also derive some additional strengthenings of Theorem  \ref{thm-main-ccd} in the case $n\geq \frac12 |S'|$ is very large. In particular, we will achieve the same strengthened conclusions recently guaranteed in \cite{IttII} under a different $n$ is large assumption  (Theorem \ref{thm-partition-thm-equi}). This, in turn, will allow us to derive additional  information for $S$, in particular, when $|S|=2n$ with  $|\Sigma_n(S)|\leq n+1$ and $\mathsf h(S)\leq n$ (Theorem \ref{thm-special-dihedral-ample}).

\section{Partitioning Results for General $n$}\label{sec-ngen}

In this section, we will make heavy use of the arguments used to prove \cite[Theorem 14.1]{Gbook} and the following easy consequence of Kneser's Theorem (see \cite[Theorem 5.1]{Gbook}).

\begin{theirtheorem}\label{thm-multbound}
Let $G$ be an abelian group, and let $A,\,B\subseteq G$ be finite, nonempty subsets. If $|A+B|< |A|+|B|-1$, then $A+(B\setminus \{x\})=A+B$ for all $x\in B$.
\end{theirtheorem}

We will also need the following observation, that follows by a routine induction on $n$.

\begin{lemma}\label{lem-local-cdtbound}
Let $G$ be an abelian group and let $A_1,\ldots,A_n\subseteq G$ be finite, nonempty subsets. Suppose $|\Sum{i=1}{j}A_i|\geq |\Sum{i=1}{j-1}A_i|+|A_j|-1$ for all $j\in [2,n]$. Then $|\Sum{i=1}{n}A_i|\geq \Sum{i=1}{n}|A_i|-n+1$. Moreover, if $|\Sum{i=1}{n}A_i|=\Sum{i=1}{n}|A_i|-n+1$, then  $|\Sum{i=1}{j}A_i|= |\Sum{i=1}{j-1}A_i|+|A_j|-1$ for all $j\in [2,n]$.
\end{lemma}

Let $G$ be an abelian group, let $X\subseteq G$ be a nonempty subset and let $S\in \Fc(G)$ be a sequence. A setpartition $\mathscr A=A_1\bdot\ldots\bdot A_n$ with $\mathsf S(\mathscr A)\mid S$ will be  called \emph{maximal relative to $X$}  if any setpartition $\mathscr B=B_1\bdot\ldots\bdot B_n$ with $\mathsf S(\mathscr B)\mid S$,  \ $|\mathsf S(\mathscr B)|=|\mathsf S(\mathscr A)|$ and $X+\Sum{i=1}{n}A_i\subseteq X+\Sum{i=1}{n}B_i$ has $X+\Sum{i=1}{n}A_i=X+\Sum{i=1}{n}B_i$. We simply say $\mathscr A$ is maximal relative to $X$ if this is the case with $S=\mathsf S(\mathscr A)$.

\begin{lemma}\label{lem-modulo-seed} Let $G$ be an abelian group, let $n\geq 1$, let $X\subseteq G$ be a finite, nonempty subset, let $S\in \Fc(G)$ be a sequence, let $\mathscr A=A_1\bdot\ldots\bdot A_n$ be a setpartition with $\mathsf S(\mathscr A)\mid S$ maximal  relative to $X$, and let $H=\mathsf H(X+\Sum{i=1}{n}A_i)$. Suppose \be\label{lem-hyper}|X+\Sum{i=1}{n}A_i|<|X|+\Sum{i=1}{n}|A_i|-n.\ee Then there exists a setpartition $\mathscr B=B_1\bdot\ldots\bdot B_n$ with $\mathsf S(\mathscr B)\mid S$, $|\mathsf S(\mathscr B)|=|\mathsf S(\mathscr A)|$ and $X+\Sum{i=1}{n}B_i=X+\Sum{i=1}{n}A_i$ such that  $\supp(\mathsf S(\mathscr B)^{[-1]}\bdot S)\subseteq Z$  and $|(y+H)\cap B_i|\leq 1$ for all $y\in G\setminus Z$ and $i\in [1,n]$, where  $Z=\bigcap_{i=1}^n(B_i+H)$.
\end{lemma}

\begin{proof}
In view of \eqref{lem-hyper} and Kneser's Theorem, we conclude that $H$ is nontrivial. Consider an arbitrary setpartition  $\mathscr B=B_1\bdot\ldots\bdot B_n$  with $\mathsf S(\mathscr B)\mid S$,  $|S(\mathscr B)|=|\mathsf S(\mathscr A)|$ and $X+\Sum{i=1}{n}A_i\subseteq X+\Sum{i=1}{n}B_i$. Then $X+\Sum{i=1}{n}B_i=X+\Sum{i=1}{n}A_i$ since $\mathscr S(\mathscr A)\mid S$ with $\mathscr A$ maximal relative to $X$. Let $Z=\bigcap_{i=1}^n(B_i+H)$.
In view of \eqref{lem-hyper} and Lemma \ref{lem-local-cdtbound}, there must be some $j\in [1,n]$ such that $|X+\Sum{i=1}{j}B_i|<|X+\Sum{i=1}{j-1}B_i|+|B_j|-1$, in which case Theorem \ref{thm-multbound} implies \be\label{pullout-o}X+\Sum{i=1}{j-1}B_i+(B_j\setminus\{x\})=X+\Sum{i=1}{j}B_i\quad\mbox{ for all $x\in B_j$}.\ee
Since $|X+\Sum{i=1}{j}B_i|<|X+\Sum{i=1}{j-1}B_i|+|B_j|-1$, Kneser's Theorem implies that $|(x+H')\cap B_j|\geq 2$ for every $x\in B_j$, where $H'=\mathsf H(X+\Sum{i=1}{j}B_i)\leq H$. In particular, \be\label{H2}|(x+H)\cap B_j|\geq 2\quad\mbox{ for all $x\in B_j$}.\ee

Now further restrict $\mathscr B$ by assuming   $\Sum{i=1}{n}|\phi_H(B_i)|$ is maximal (subject to the defining condition for $\mathscr B$). Then we must have $B_j\subseteq Z$, where $j\in [1,n]$ is the index defined above. Indeed, if this fails, then there is some $x\in B_j\setminus Z$, and thus also some $k\in [1,n]$ with $\phi_H(x)\notin \phi_H(B_k)$ by definition of $Z$. We can then remove $x$ from $B_j$ and place it in $B_k$ to yield a new setpartition $\mathcal B'=B'_1\bdot\ldots\bdot B'_n$, where $B'_j=B_j\setminus\{x\}$, $B'_k=B_k\cup \{x\}$ and $B'_i=B_i$ for $i\neq j,k$. In view of \eqref{pullout-o}, we have $X+\Sum{i=1}{n}A_i=X+\Sum{i=1}{n}B_i\subseteq X+\Sum{i=1}{n}B'_i$, while in view of \eqref{H2} and $\phi_H(x)\notin \phi_H(B_k)$, we have $\Sum{i=1}{n}|\phi_H(B'_i)|=\Sum{i=1}{n}|\phi_H(B_i)|+1$, and now $\mathscr B'$ contradicts the maximality of $\Sum{i=1}{n}|\phi_H(B_i)|$ for $\mathscr B$. Therefore \be\label{H3} B_j\subseteq Z.\ee

\subsection*{Claim A} $|(y+H)\cap B_i|\leq 1$ for all $y\in G\setminus Z$ and $i\in [1,n]$.

\begin{proof}
Assume by contradiction there is some $k\in [1,n]$ and $y\in B_k\setminus Z$ with $|(y+H)\cap B_k|\geq 2$. 
Let $\mathscr C=C_1\bdot\ldots\bdot C_n$ be a setpartition with $\mathsf S(\mathscr C)=\mathsf S(\mathscr B)$, $X+\Sum{i=1}{n}C_i=X+\Sum{i=1}{n}A_i$, and $C_i\setminus Z=B_i\setminus Z$ and $\phi_H(C_i)=\phi_H(B_i)$ for all $i$, such that $|C_k|$ is maximal. Since $\phi_H(C_i)=\phi_H(B_i)$ for all $i$, we still have $\Sum{i=1}{n}|\phi_H(C_i)|$ maximal, while $Z=\bigcap_{i=1}^n(B_i+H)=\bigcap_{i=1}^n(C_i+H)$. Thus, let $j'\in [1,n]$ be an index so that $C_{j'}$ satisfies \eqref{pullout-o}, \eqref{H2} and \eqref{H3} for $\mathscr C$ (in place of $B_j$).

Suppose $C_{j'}\subseteq C_k$. Since $y\notin Z$ but $C_{j'}\subseteq Z$ (by \eqref{H3}), we actually have $C_{j'}\subseteq C_{k}\setminus \{y\}$, in which case \be\label{H4}C_{j'}+C_k\subseteq (C_{j'}\cup \{y\})+(C_k\setminus \{y\}).\ee In such case, we can define a new setpartition $\mathscr C'=C'_1\bdot\ldots\bdot C'_n$ by removing $y$ from $C_k$ and placing it in $C_{j'}$, so $C'_{k}=C_k\setminus \{y\}$, $C'_{j'}=C_{j'}\cup \{y\}$ and $C'_i=C_i$ for $i\neq k,j'$.
 Note $y\in B_k\setminus Z=C_k\setminus Z$.
 In view of \eqref{H4}, we have $X+\Sum{i=1}{n}A_i=X+\Sum{i=1}{n}C_i\subseteq X+\Sum{i=1}{n}C'_i$, while in view of $|(y+H)\cap C_k|=|(y+H)\cap B_k|\geq 2$ (as $y\notin Z$ and $C_k\setminus Z=B_k\setminus Z$) and $\phi_H(y)\notin \phi_H(C_{j'})$ (as $y\notin Z$ and $C_{j'}\subseteq Z$ by \eqref{H3}), we have $\Sum{i=1}{n}|\phi_H(C'_i)|=\Sum{i=1}{n}|\phi_H(C_i)|+1=\Sum{i=1}{n}|\phi_H(B_i)|+1$, so that $\mathscr C'$ contradicts the maximality of $\Sum{i=1}{n}|\phi_H(B_i)|$ for $\mathscr B$. So we instead conclude that $C_{j'}\nsubseteq C_k$. Thus, in view \eqref{H3}, it follows that there is some $x\in C_{j'}\subseteq Z$ with $x\notin C_k$.

In this case, we define a new setpartition $\mathscr C'=C'_1\bdot\ldots\bdot C'_n$ by removing $x$ from $C_{j'}$ and placing it in $C_{k}$, so $C'_{j'}=C_{j'}\setminus \{x\}$, $C'_{k}=C_{k}\cup \{x\}$ and $C'_i=C_i$ for $i\neq j',k$. In view of \eqref{pullout-o}, we have $X+\Sum{i=1}{n}A_i=X+\Sum{i=1}{n}C_i\subseteq X+\Sum{i=1}{n}C'_i$, while in view of \eqref{H2} and $x\in Z$,  we have $\phi_H(C'_i)=\phi_H(C_i)=\phi_H(B_i)$ and $C'_i\setminus Z=C_i\setminus Z=B_i\setminus Z$ for all $i$. Thus, since $|C'_k|=|C_k|+1$, we see that $\mathscr C'$ contradicts the maximality of $|C_k|$ for $\mathscr C$, completing the claim.
\end{proof}

In view of Claim A, we see that the lemma holds with the setpartition $\mathscr B$ unless there is some $y\in \supp(\mathsf S(\mathscr B)^{[-1]}\bdot S)$ with $y\notin Z$. However, if this were the case, then $\phi_H(y)\notin \phi_H(B_j)$ in view of \eqref{H3}. Define a new setpartition $\mathscr B'=B'_1\bdot\ldots\bdot B'_n$ by removing any term $x\in B_j$ from $B_j$ and placing $y$ into $B_j$ instead, so $B'_j=B_j\setminus\{x\}\cup \{y\}$ and $B'_i=B_i$ for $i\neq j$.  In view of \eqref{pullout-o}, we have  $X+\Sum{i=1}{n}A_i=X+\Sum{i=1}{n}B_i\subseteq X+\Sum{i=1}{n}B'_i$, while in view of \eqref{H2} and $\phi_H(y)\notin \phi_H(B_j)$, we have $\Sum{i=1}{n}|\phi_H(B'_i)|=\Sum{i=1}{n}|\phi_H(B_i)|+1$, in which case $\mathscr B'$ contradicts the maximality of $\Sum{i=1}{n}|\phi_H(B_i)|$ for $\mathscr B$, completing the proof.
\end{proof}

\begin{lemma}\label{lem-modulo-equiazation} Let $G$ be an abelian group, let $n\geq 1$, let $X\subseteq G$ be a finite, nonempty subset, let $\mathscr A=A_1\bdot\ldots\bdot A_n$ be a  setpartition  over $G$ maximal relative to $X$, let $H=\mathsf H(X+\Sum{i=1}{n}A_i)$ and let $Z=\bigcap_{i=1}^n(A_i+H)$. Suppose $|(y+H)\cap A_i|\leq 1$ for all $y\in G\setminus Z$ and $i\in [1,n]$, and \be\label{lem-hyp}|X+\Sum{i=1}{n}A_i|<|X|+\Sum{i=1}{n}|A_i|-n+(|H|-1).\ee Then there exists a setpartition $\mathscr B=B_1\bdot\ldots\bdot B_n$ with $\mathsf S(\mathscr B)=\mathsf S(\mathscr A)$, $X+\Sum{i=1}{n}B_i=X+\Sum{i=1}{n}A_i$ and  $Z\subseteq \bigcap_{i=1}^n(B_i+H)$ such that  $|B_i\setminus Z|\leq 1$ for all $i$. 
\end{lemma}

\begin{proof}
In view of \eqref{lem-hyp} and Kneser's Theorem, we conclude that $H$ is nontrivial. Consider a setpartition  $\mathscr B=B_1\bdot\ldots\bdot B_n$  with \begin{align}\label{wilecate}&\mathsf S(\mathscr B)=\mathsf S(\mathscr A), \quad X+\Sum{i=1}{n}B_i=X+\Sum{i=1}{n}A_i,\quad Z=\bigcap_{i=1}^n(A_i+H)\subseteq \bigcap_{i=1}^n(B_i+H)\quad\und\quad \\&|(y+H)\cap B_i|\leq 1\quad\mbox{ for all $y\in G\setminus Z$ and $i\in [1,n]$}.\nn\end{align}
Since $\mathscr A$ satisfies these conditions, it follows that such a setpartition $\mathscr B$ exists. Let $e=\Sum{i=1}{n}|B_i\setminus Z|\geq 0$.  Let $I_e\subseteq [1,n]$ be all those indices $i\in [1,n]$ with $B_i\setminus Z$ nonempty, and let $I_Z\subseteq [1,n]$ be all those indices $i\in [1,n]$ with $B_i\subseteq Z$.  By re-indexing the $B_i$, we can w.l.o.g. assume $I_Z=[1,m]$ and $I_e=[m+1,n]$.

Suppose $|X+\Sum{i=1}{j}B_i|\geq |X+\Sum{i=1}{j-1}B_i|+|B_j|-1$ for all $j\in [1,m]$. Then Lemma \ref{lem-local-cdtbound} implies $|X+\Sum{i=1}{m}B_i|\geq |X|+\Sum{i=1}{m}|B_i|-m$. Kneser's Theorem implies \begin{align}\label{toyl}|\Big(X+\Sum{i=1}{m}B_i\Big)+\Sum{i=m+1}{n}B_i|&\geq |X+\Sum{i=1}{m}B_i|+\Sum{i=m+1}{n}|B_i+H|-(n-m)|H|\\&\nn\geq  |X+\Sum{i=1}{m}B_i|+\Sum{i=m+1}{n}|B_i|+e(|H|-1)-(n-m)|H|,\end{align} with the second inequality in view of the final condition in \eqref{wilecate}. Combined with the previous estimate for $|X+\Sum{i=1}{m}B_i|$, we find \be\label{start}|X+\Sum{i=1}{n}B_i|\geq |X|+\Sum{i=1}{n}|B_i|-n|H|+(m+e)(|H|-1).\ee
Note that $e\geq |I_e|= n-m$. If equality holds, then $|B_i\setminus Z|\leq 1$ follows for all $i$, completing the proof. Therefore we can instead assume $e\geq n-m+1$, which combined with \eqref{start} yields $|X+\Sum{i=1}{n}A_i|=|X+\Sum{i=1}{n}B_i|\geq |X|+\Sum{i=1}{n}|B_i|-n+(|H|-1)=|X|+\Sum{i=1}{n}|A_i|-n+(|H|-1)$, contrary to hypothesis. So we may instead assume there is some $j\in [1,m]$ with $|X+\Sum{i=1}{j}B_i|< |X+\Sum{i=1}{j-1}B_i|+|B_j|-1$. In particular, this argument shows that $I_Z$ is nonempty for any setpartition satisfying \eqref{wilecate},
and  Theorem \ref{thm-multbound} ensures that \be\label{remove}X+\Sum{i=1}{j-1}B_i+(B_j\setminus \{x\})=X+\Sum{i=1}{j}B_i\quad\mbox{ for all $x\in B_j$}.\ee In particular, $|B_j|\geq 2$. Let $K=\mathsf H(X+\Sum{i=1}{j}B_i)\leq H$. If $|(y+H)\cap B_j|=1$ for some $y\in G$, then $|(y+K)\cap B_j|=1$ as well, whence Kneser's Theorem implies $|\big(X+\Sum{i=1}{j-1}B_i\Big)+B_j|\geq |X+\Sum{i=1}{j-1}B_i|+|B_j+K|-|K|\geq  |X+\Sum{i=1}{j-1}B_i|+|B_j|-1$, contrary to the definition of $j$. Therefore we instead conclude that \be\label{Hmany} |(y+H)\cap B_j|\geq 2\quad\mbox{ for all $y\in B_j$}.\ee


Now assume our setpartition $\mathscr B$ satisfying \eqref{wilecate} is chosen such that
\begin{itemize}
\item[M1.] $|I_e|$ is maximal (subject to \eqref{wilecate}),
\item[M2.] $\Summ{i\in I_e}|B_i|$ is maximal (subject to \eqref{wilecate} and M1).
\end{itemize}
If $|B_i\setminus Z|=1$ for every $i\in I_e=[m+1,n]$, then the setpartition $\mathscr B$  satisfies the conditions of the lemma. Therefore we may assume there is some $k\in I_e= [m+1,n]$ with distinct $y_1,y_2\in B_k\setminus Z$. Since $j\in [1,m]=I_Z$, we have $B_j\subseteq Z$. Thus $y_1,y_2\notin B_j$.

Suppose $B_s\subseteq B_k$ for some $s\in I_Z$. Then $B_s\subseteq B_k\setminus \{y_1\}$ and $B_s+B_k\subseteq (B_s\cup \{y_1\})+(B_k\setminus \{y_1\})$, the former as $y_1\notin Z$ but $B_s\subseteq Z$ as $s\in  I_Z$.
In such case, define a new setpartition $\mathscr B'=B'_1\bdot\ldots\bdot B'_n$ by setting $B'_s=B_s\cup \{y_1\}$, $B'_k=B_k\setminus \{y_1\}$ and $B'_i=B_i$ for $i\neq s,k$.
Then $B_s+B_k\subseteq (B_s\cup \{y_1\})+(B_k\setminus \{y_1\})=B'_s+B'_k$ ensures that $X+\Sum{i=1}{n}A_i=X+\Sum{i=1}{n}B_i\subseteq X+\Sum{i=1}{n}B'_i$, and equality must hold as $\mathscr A$ is maximal relative to $X$.
By definition, $\mathsf S(\mathscr B')=\mathsf S(\mathscr B)=\mathsf S(\mathscr A)$. Since $B_s\subseteq Z$, we have still have $|(y+H)\cap B'_i|\leq 1$ for all $i$ and $y\in G\setminus Z$, while $Z\subseteq \bigcap_{i=1}^{n}B'_i$ follows since $y_1\notin Z$.
Thus $\mathscr B'$ satisfies \eqref{wilecate}.
%
Since $y_1\notin Z$ and $y_2\in B'_k\setminus Z$, we see $I_e\cup \{j\}\subseteq [1,n]$ is the subset of indices $i\in [1,n]$ for which $B'_i\setminus Z$ is nonempty, meaning $\mathscr B'$ contradicts the maximality condition M1 for $\mathscr B$. So we instead assume $B_s\not\subseteq B_k$ for all $s\in I_Z$. In particular, $B_j\not\subseteq B_k$, meaning there is some $x\in B_j\setminus B_k$.

In this case, define a new setpartition $\mathscr B'=B'_1\bdot\ldots\bdot B'_n$ by setting $B'_j=B_j\setminus \{x\}$, $B'_k=B_k\cup \{x\}$ and $B'_i=B_i$ for $i\neq j,k$. In view of \eqref{remove}, we have $X+\Sum{i=1}{n}A_i=X+\Sum{i=1}{n}B_i\subseteq X+\Sum{i=1}{n}B'_i$, and equality must hold as $\mathscr A$ is maximal relative to $X$. By definition, $\mathsf S(\mathscr B')=\mathsf S(\mathscr B)=\mathsf S(\mathscr A)$. Since $x\in B_j\subseteq Z$, we have still have $|(y+H)\cap B'_i|\leq 1$ for all $i$ and $y\in G\setminus Z$, while $Z\subseteq \bigcap_{i=1}^{n}B'_i$ follows in view of \eqref{Hmany}.
Thus $\mathscr B'$ satisfies \eqref{wilecate}.
Since $x\in B_j\subseteq Z$, we see $I_e\subseteq [1,n]$ is still the subset of indices $i\in [1,n]$ for  which $B'_i\setminus Z$ is nonempty, meaning $\mathscr B'$ satisfies M1. However, since $|B'_k|=|B_k|+1$, $k\in I_e$ and $j\notin I_e$, the maximality of $\Summ{i\in I_e}|B_i|$ for $\mathscr B$ is contradicted by $\mathscr B'$.
\end{proof}


\begin{lemma}\label{lem-subsums=sumset}
Let $G$ be an abelian group, let $n\geq 0$, let $X\subseteq G$ be a finite, nonempty subset,   let $S\in \Fc(G)$ be a sequence, and let  $\mathscr A=A_1\bdot\ldots\bdot A_n$ be a setpartition with $\mathsf S(\mathscr A)\mid S$,  $\supp(\mathsf S(\mathscr A)^{[-1]}\bdot S)\subseteq Z$,  and  $|A_i\setminus Z|\leq 1$ for all $i$, where $Z=Z+H\subseteq \bigcap_{i=1}^n(A_i+H)$ and $H\leq \mathsf H(X+\Sum{i=1}{n}A_i)$. Then the following hold.
\begin{itemize}
\item[1.] $X+\Sigma_{n}(S)=X+\Sum{i=1}{n}A_i$.
\item[2.] If $Z=g+H$ for some $g\in G$, then $X+\Sigma_{\ell}(S)=X+\Sum{i=1}{n}A_i+(\ell-n)g$ for any $\ell\in [n,n+|\mathsf S(\mathscr A)^{[-1]}\bdot S|]$.
    \end{itemize}
\end{lemma}

\begin{proof}
1. Note $X+\Sum{i=1}{n}A_i\subseteq X+\Sigma_n(S)$ holds trivially. Let $T\mid S$ with $|T|=n$ be arbitrary. Since $X+\Sum{i=1}{n}A_i$ is $H$-periodic by hypothesis, to establish the reverse inclusion, it suffices to show $\sigma(\phi_H(T))\in \Sum{i=1}{n}\phi_H(A_i)$. Write $T=x_1\bdot\ldots\bdot x_s\bdot y_{s+1}\bdot\ldots\bdot y_n$, with the $x_i$ the terms of $T$ with $x_i\in G\setminus  Z$, and the $y_i$ the terms of $T$ with $y_i\in Z$. In view of $\supp(\mathsf S(\mathscr A)^{[-1]}\bdot S)\subseteq Z$ and  $|A_i\setminus Z|\leq 1$ for all $i$, we can re-index the $A_i$ so that $x_i\in A_i$ for $i\in [1,s]$. But now, since $y_j\in Z=Z+H\subseteq \bigcap_{i=1}^{n}(A_i+H)\subseteq A_j+H$ for all $j\geq s+1$, it follows that $\sigma(\phi_H(T))=\phi_H(x_1)+\ldots+\phi_H(x_s)+\phi_H(y_{s+1})+\ldots+\phi_H(y_n)\in \Sum{i=1}{n}\phi_H(A_i)$, completing  Item 1.

2. By translating all terms of $S$ appropriately by $-g$, we  can w.l.o.g. assume $g=0$, whence $H\subseteq \bigcap_{i=1}^n(A_i+H)$. In particular, $H\cap A_i\neq\emptyset$ and $|A_i\setminus H|\leq 1$ for all $i$.
Since $X+\Sum{i=1}{n}A_i$ is $H$-periodic with $\mathsf S(\mathscr A)\mid S$, $\supp(\mathsf S(\mathscr A)^{[-1]}\bdot S)\subseteq g+H=H$ and $n\leq \ell\leq n+|\mathsf S(\mathscr A)^{[-1]}\bdot S|$, we trivially have $X+\Sum{i=1}{n}A_i=X+\Sum{i=1}{n}A_i+(\ell-n)g\subseteq X+\Sigma_{\ell}(S)$. To show the reverse inclusion, let  $T=g_1\bdot\ldots\bdot g_\ell$ be an arbitrary $\ell$-term subsequence of $S$. Since $\supp(\mathsf S(\mathscr A)^{[-1]}\bdot S)\subseteq g+H=H$ and $|A_i\setminus H|\leq 1$ and  for all $i$, there are at most $n$ non-zero terms in $\phi_H(T)$, and by re-indexing, we can w.l.o.g. assume $\phi_H(g_i)=0$ for $i>n$. Then, since $H\cap A_i\neq \emptyset$, \ $\supp(\mathsf S(\mathscr A)^{[-1]}\bdot S)\subseteq H$ and $|A_i\setminus H|\leq 1$ for all $i$, it follows that $\phi_H(g_1)+\ldots+\phi_H(g_\ell)=\phi_H(g_1)+\ldots+\phi_H(g_n)\in \Sum{i=1}{n}\phi_H(A_i)$. Hence, since $X+\Sum{i=1}{n}A_i$ is $H$-periodic, we conclude that $X+\sigma(T)=X+g_1+\ldots+g_\ell\subseteq X+\Sum{i=1}{n}A_i$. Since $T$ was an arbitrary $\ell$-term subsequence of $S$, this establishes the reverse inclusion $X+\Sigma_\ell(S)\subseteq X+\Sum{i=1}{n}A_i$.
\end{proof}

Let $G$ be an abelian group and $A\subseteq G$ a subset. We say  $A$ is \emph{quasi-periodic} if there is  a subset $A_\emptyset\subseteq A$ such that $A\setminus A_\emptyset$ is nonempty and periodic with $A_\emptyset$ contained in a $\mathsf H(A\setminus A_\emptyset)$-coset. If $H\leq G$ is a nontrivial subgroup, then an \emph{$H$-quasi-periodic decomposition} is a partition $A=(A\setminus A_\emptyset)\cup A_\emptyset$ with $A_\emptyset$ a subset of an $H$-coset and $A\setminus A_\emptyset$ $H$-periodic (or empty). It is \emph{reduced} if $A_\emptyset$ is not quasi-periodic.
As is easily derived,
\be\label{dumpling}\mathsf H(A)=\mathsf H(A_\emptyset)\leq H\quad\mbox{ when $\emptyset\neq A_\emptyset\subset A_{\emptyset}+H$}.\ee
If $A_\emptyset$ is quasi-periodic, as exhibited by  $A'_\emptyset\subseteq A_\emptyset$, then $A=(A\setminus A'_\emptyset)\cup A'_\emptyset$ is a quasi-periodic decomposition with $A'_\emptyset\subset A_\emptyset$. Every finite set $A\subseteq G$ has a reduced quasi-periodic decomposition, and this decomposition is unique unless $A\cup \{\alpha\}$ is periodic for some $\alpha\notin A$ (see \cite[Proposition 2.1]{kst+quasi}). The Kemperman Structure Theorem \cite[Theorem 9.1]{Gbook} implies that, if $A,\,B\subseteq G$ are finite, nonempty subsets with $|A+B|=|A|+|B|-1$ and either $A+B$ aperiodic or containing a unique expression element, then there are $H$-quasi-periodic decompositions $A=(A\setminus A_\emptyset)\cup A_\emptyset$, $B=(B\setminus B_\emptyset)\cup B_\emptyset$ and $A+B=\Big((A+B)\setminus (A_\emptyset +B_\emptyset)\Big)\cup (A_\emptyset+B_\emptyset)$ with the pair $(A_\emptyset,B_\emptyset)$ satisfying one of four possible structural types (I)--(IV), each with explicitly defined restrictions on where, and how many, unique expression elements there are.
We will make use of this theory,  referencing the  details regarding Kemperman's Critical Pair Theory  rather than repeating the rather lengthy statements and details  here.

\begin{lemma}
\label{lemma-kst-apuncunctured}
Let $G$ be an abelian group, let $H\leq G$ be a subgroup with $|H|\geq 3$, and let $Y\subseteq G$ be a finite subset such that $Y\setminus \{y_0\}$ is $H$-periodic (or empty) for some $y_0\in Y$.
\begin{itemize}
\item[1.] $Y$ is aperiodic with $Y=(Y\setminus \{y_0\})\cup \{y_0\}$ its unique reduced quasi-periodic decomposition.
\item[2.] If there is a $K$-quasi-periodic decomposition $Y=Y_1\cup Y_0$, then  $y_0\in Y_0$ and $Y_0\setminus \{y_0\}$ is $H$-periodic. Moreover, if $|Y_0|\geq 2$, then $H\leq K$.
\item[3.] If $A,\,B\subseteq G$ with $A+B=Y$ and $|A+B|=|A|+|B|-1$, then  there are $a_0\in A$ and $b_0\in B$ such that $A\setminus \{a_0\}$ and $B\setminus \{b_0\}$ are $H$-periodic with $a_0+b_0=y_0$.
\end{itemize}
\end{lemma}

\begin{proof}
We may w.l.o.g. assume $H=\mathsf H(Y\setminus \{y_0\})$. We may also assume $Y\setminus \{y_0\}$ is nonempty, else Items 1--3 all hold trivially. Item 1 follows from \cite[Proposition 2.1]{kst+quasi} and \cite[Comment c.6]{kst+quasi}.

2. Suppose $Y=Y_1\cup Y_0$ is a $K$-quasi-periodic decomposition and let $Y_0=(Y_0\setminus Y_\emptyset)\cup Y_\emptyset$ be a reduced quasi-periodic decomposition of $Y_0$. Then $(Y\setminus Y_\emptyset)\cup Y_\emptyset$ is a reduced quasi-periodic decomposition of
$Y$ with either  $\mathsf H(Y\setminus Y_\emptyset)=\mathsf H(Y_0\setminus Y_\emptyset)\leq \la Y_0-Y_0\ra\leq K$ or $Y_0= Y_\emptyset$ (by \eqref{dumpling}).
Hence Item 1 ensures that $Y_\emptyset=\{y_0\}$ with $H=\mathsf H(Y\setminus \{y_0\})=\mathsf H(Y\setminus Y_\emptyset)\leq K$ if $|Y_0|\geq 2$. If $Y_0=\{y_0\}=Y_\emptyset$, then $Y_1=Y\setminus\{y_0\}=Y\setminus Y_\emptyset$ is $H$-periodic. Otherwise, $H\leq K$ ensures $Y_1$ is $H$-periodic, while $Y\setminus Y_\emptyset=Y\setminus \{y_0\}$ is $H$-periodic by hypothesis. In either case,
$Y_1$ and $Y\setminus Y_\emptyset$ are both $H$-periodic, and it follows that $(Y\setminus Y_\emptyset)\setminus Y_1=Y_0\setminus Y_\emptyset=Y_0\setminus \{y_0\}$ is also $H$-periodic. Item 2 now follows.

3.  Suppose $A+B=Y$ and $|A+B|=|A|+|B|-1$. Since $Y=A+B$ is aperiodic by Item 1, we can directly apply the Kemperman Structure Theorem \cite[Proposition 2.1]{kst+quasi} to $A+B$ yielding associated $K$-quasi-periodic decompositions  $A=(A\setminus A_\emptyset)\cup A_\emptyset$, $B=(B\setminus B_\emptyset)\cup B_\emptyset$ and $Y=(Y\setminus (A_\emptyset +B_\emptyset))\cup (A_\emptyset+B_\emptyset)$. If the pair $(A_\emptyset, B_\emptyset)$ has type (IV), then $Y\cup \{\beta\}$ is periodic for some $\beta\in G\setminus Y$. In such case, any reduced quasi-periodic decomposition $Y=Y_1\cup Y_0$ must have $|Y_0|\geq 2$ or $|\mathsf H(Y_1)|=2$ (cf. \cite[Proposition 2.1]{kst+quasi}), contrary to Item 1.
If the pair $(A_\emptyset, B_\emptyset)$ has type (III), then $Y$ is periodic, contrary to Item 1.
If the pair $(A_\emptyset, B_\emptyset)$ has type (II), then either $Y\cup \{\beta\}$ is periodic for some $\beta\in G$, yielding the same contradiction as before, or else $Y=(Y\setminus (A_\emptyset +B_\emptyset))\cup (A_\emptyset+B_\emptyset)$ is a reduced quasi-periodic decomposition of $Y$ (by  \cite[Comment c.3]{kst+quasi}) with $|A_\emptyset+B_\emptyset|\geq 3$, again contrary to Item 1.
We are left to conclude that $(A_\emptyset, B_\emptyset)$ has type (I), so w.l.o.g. $|A_\emptyset|=1$, say  with $A_\emptyset=\{a_0\}$. Hence $A+B=Y$ is a union of $a_0+B$ with a $K$-periodic set. In particular, $(Y\setminus (a_0+B_\emptyset))\cup (a_0+B_\emptyset)$ is a $K$-quasi-periodic decomposition of $Y$. If $|B_\emptyset|\geq 2$, then Item 2 yields $H\leq K$ with $y_0\in a_0+B_0$ and $(a_0+B_\emptyset)\setminus \{y_0\}$ $H$-periodic. Letting $b_0=y_0-a_0\in B_0$, Item 3 follows.
Therefore instead assume $|B_\emptyset|=1$, say $B_\emptyset=\{b_0\}$. Moreover, in view of \cite[Proposition 2.2]{kst+quasi}, we can assume $a_0+b_0\in A+B$ is the only unique expression element.

In this case $\Big((A+B)\setminus \{a_0+b_0\}\Big)\cup \{a_0+b_0\}$ is a reduced $K$-quasi-periodic decomposition, in which case Item 1 implies $a_0+b_0=y_0$ and $(A+B)\setminus \{a_0+b_0\}=Y\setminus \{y_0\}$.
Since $(y_0+H)\cap (A+B)=(y_0+H)\cap Y=\{y_0\}$ with $y_0=a_0+b_0$, we must have $(a_0+H)\cap A=\{a_0\}$ and $(b_0+H)\cap B=\{b_0\}$.
If $|A_0|=|B_0|=1$, then Item 3 follows trivially. Therefore we can instead w.l.o.g. assume $|A_0|\geq 2$. Since $a_0+b_0=y_0\in A+B=Y$ is the only unique expression, we have $(A\setminus \{a_0\})+B=Y\setminus \{y_0\}$  with $\mathsf H((A\setminus \{a_0\})+B)=\mathsf H(Y\setminus \{y_0\})=H$. Kneser's Theorem now implies \be\label{chickdal}|A|+|B|-2=|(A\setminus \{a_0\})+B|\geq |(A\setminus \{a_0\})+H|+|B+H|-|H|\geq |A\setminus \{a_0\}|+|B|-1,\ee with the latter inequality in view of $(b_0+H)\cap B=\{b_0\}$. Thus we must have equality in \eqref{chickdal}. In particular, equality holding in the second inequality in \eqref{chickdal} forces $(A\setminus \{a_0\})+H=A\setminus \{a_0\}$ and $(B\setminus \{b_0\})+H=B\setminus \{b_0\}$, i.e., $A\setminus \{a_0\}$ and $B\setminus \{b_0\}$ are $H$-periodic, completing  Item 3.
\end{proof}

\begin{lemma}\label{lem-kt}
Let $G$ be an abelian group and let $A,\,B\subseteq G$ be finite, nonempty subsets with $H=\mathsf H(A+B)=\mathsf H(B)$. Then $|(A\cup \{x\})+B|\geq |(A\cup \{x\})+H|+|B+H|-|H|$ for any $x\in G$.
\end{lemma}

\begin{proof}
By hypothesis, $\phi_H(A)+\phi_H(B)$ is aperiodic. Thus $(\phi_H(A)\cup \{\phi_H(x)\})+\phi_H(B)$ is either still aperiodic, in which case Kneser's Theorem implies  $|\phi_H(A\cup \{x\})+\phi_H(B)|\geq |\phi_H(A\cup \{x\})|+|\phi_H(B)|-1$, or else it is periodic, and thus strictly contains the aperiodic subset $\phi_H(A)+\phi_H(B)$. In such case, applying Kneser' Theorem to $\phi_H(A)+\phi_H(B)$ yields $|\phi_H(A\cup \{x\})+\phi_H(B)|\geq |\phi_H(A)|+|\phi_H(B)|\geq |\phi_H(A\cup \{x\})|+|\phi_H(B)|-1$. In either case, since $(A\cup \{x\})+B$ is $H$-periodic in view of $H=\mathsf H(B)$, the desired conclusion follows by multiplying the inequality by $|H|$.

\end{proof}

\begin{lemma}\label{lem-genn-equiazation} Let $G$ be an abelian group, let $n\geq 1$, let $X\subseteq G$ be a finite, nonempty subset, let $S\in \Fc(G)$ be a sequence, and let $\mathscr A=A_1\bdot\ldots\bdot A_n$ be a setpartition with $\mathsf S(\mathscr A)=S$ and $|X+\Sum{i=1}{n}A_i|\geq
\min\{|S|-n+|X|,\; |X+\Sigma_n(S)|\}$.
 Then one of the following holds.
\begin{itemize}
\item[1.] $n=2$, $X\setminus \{\beta\}$ and $A_1\cap A_2$ are $K$-periodic, $(A_1\cup A_2)\setminus (A_1\cap A_2)$ is a $K$-coset, $X+\Sigma_n(S)=X+A_1+A_2$ is aperiodic and $|X+\Sigma_n(S)|=|S|-n+|X|$,  for some $\beta\in X$ and  $K\leq G$ with $K\cong (\Z/2\Z)^2$.
\item[2.] There exists an \emph{equitable} setpartition $\mathscr B=B_1\bdot\ldots\bdot B_n$ with $\mathsf S(\mathscr B)=S$ and
   $|X+\Sum{i=1}{n}B_i|\geq \min\{|S|-n+|X|,\,|X+\Sigma_n(S)|\}$. Moreover, if $|X+\Sigma_n(S)|\leq |S|-n+|X|$ and $|A_i\setminus Z|\leq 1$ for all $i$, where  $H=\mathsf H(X+\Sum{i=1}{n}A_i)$ and  $Z=\bigcap_{i=1}^n(A_i+H)$, then  $Z=\bigcap_{i=1}^n(B_i+H)$, $X+\Sum{i=1}{n}B_i=X+\Sum{i=1}{n}A_i=X+\Sigma_n(S)$ and $|B_i\setminus Z|\leq 1$ for all $i$.
\end{itemize}
\end{lemma}

\begin{proof}
 If $|S|=\Sum{i=1}{n}|A_i|\leq n+1$ or $n=1$, then $\mathscr A$ is trivially equitable, and Item 2 follows taking $\mathscr B$ to be $\mathscr A$. Therefore we may assume $|S|\geq n+2$ and $n\geq 2$. In particular, $G$ is nontrivial.
Let $\mathscr A=A_1\bdot\ldots\bdot A_n$ be an arbitrary setpartition
with $\mathsf S(\mathscr A)=S$ and  \be\label{lupper} |X+\Sum{i=1}{n}A_i|\geq \min\{ |X|+\Sum{i=1}{n}|A_i|-n,|X+\Sigma_n(S)|\}.\ee Note $\mathscr A$ exists by hypothesis.
If, for the original setpartition $\mathscr A$, we have $|X+\Sigma_n(S)|\leq |S|-n+|X|$ and $|A_i\setminus Z|\leq 1$ for all $i$, where $H=\mathsf H(X+\Sum{i=1}{n}A_i)$ and $Z=\bigcap_{i=1}^n(A_i+H)$, then fix the set $Z\subseteq G$ and only consider setpartitions $\mathscr A$ also satisfying
\be\label{brunch} Z\subseteq \bigcap_{i=1}^n(A_i+H)\quad\und\quad |A_i\setminus Z|\leq 1\quad\mbox{ for all $i$},\ee for the fixed set $Z$. Otherwise, simply let $Z=H=G$. In either case, let $e=\Sum{i=1}{n}|A_i\setminus Z|$.

Note $|X+\Sigma_n(S)|\leq |S|-n+|X|=|X|+\Sum{i=1}{n}|A_i|-n$ and $|A_i\setminus Z|\leq 1$ for all $i$ combined with \eqref{lupper} imply $X+\Sum{i=1}{n}A_i=X+\Sigma_n(S)$, combined with $\Sum{i=1}{n}|A_i|=|S|\geq n+2$ imply $Z$ is nonempty, and combined with the definition of $e$ imply $e\leq n$.
Moreover, if $e=n$, then the $n$ terms from the sets $A_i\setminus Z$ for $i=1,\ldots,n$ cannot all be equal modulo $H$, else they would be included in the set $Z$ by definition. What this means is that an arbitrary setpartition $\mathscr A$ satisfying \eqref{lupper} and \eqref{brunch} must have $X+\Sum{i=1}{n}A_i=X+\Sigma_n(S)$ and $Z=\bigcap_{i=1}^n(A_i+H)$, so that the quantities  $H=\mathsf H(X+\Sum{i=1}{n}A_i)=\mathsf H(X+\Sigma_n(S))$ and $Z=\bigcap_{i=1}^n(A_i+H)$ remain invariant as we range over all setpartitions satisfying \eqref{lupper} and \eqref{brunch}. This is also trivially the case when $Z=H=G$.
Now choose a setpartition $\mathscr A$ with $\mathsf S(\mathscr A)=S$ satisfying \eqref{lupper} and  \eqref{brunch}  that is as equitable as possible, meaning  one such that
\be\label{min-squares-gen}\Sum{i=1}{n}|A_i|^2\quad\mbox{ is minimal}.\ee

Assume by contradiction that $\mathscr A$ is not equitable, so $$m:=\min_{i\in [1,n]}\{|A_i|\}\leq \max_{i\in [1,n]}\{|A_i|\}-2.$$ Note this ensures that $H$ is nontrivial, lest $|A_i|\in \{|Z|,|Z|+1\}$ for all $i$. Let $I_{m}\subseteq [1,n]$ be the subset of indices $i\in [1,n]$ with $|A_i|=m$, let $I_{m+1}\subseteq [1,n]$ be the subset of indices $i\in [1,n]$ with $|A_i|=m+1$, let $I_{m+2}\subseteq [1,n]$ be the subset of indices $i\in [1,n]$ with $|A_i|\geq m+2$, let
$I_Z\subseteq [1,n]$ be all indices $i\in [1,n]$ with $A_i\subseteq Z$, let $I_e\subseteq [1,n]$ be all indices $i\in [1,n]$ with $A_i\setminus Z$ nonempty, and let $$J_Z=(I_m\cup I_{m+1})\cap I_Z\quad\und\quad J_e=(I_m\cup I_{m+1})\cap I_e.$$

Since $\mathscr A$ is not equitable, $I_{m+2}$ and $I_m$ are  nonempty. Consider $k\in I_{m+2}$ and $s\in I_m$. Since $k\in I_{m+2}$ and  $|A_k\setminus Z|\leq 1$, we have $|A_k\cap Z|\geq |A_k|-1\geq m+1$. Let $\alpha_1,\ldots,\alpha_r\in A_k\cap Z$ be those elements contained in $Z$ with $(\alpha_i+H)\cap A_k=\{\alpha_i\}$, and let $Z_1=\{\alpha_1,\ldots,\alpha_r\}+H\subseteq Z$. Then $|A_k\cap (Z\setminus Z_1)|= |A_k\cap Z|-r\geq |A_k|-1-r\geq m+1-r$. Since $s\in I_m$ and $Z\subseteq A_s+H$, it follows that $|(Z\setminus Z_1)\cap A_s|\leq |A_s\cap Z|-r\leq |A_s|-r=m-r$, in which case the pigeonhole principle guarantees there is some $y\in (A_k\setminus A_s)\cap Z$ with $|(y+H)\cap A_k|\geq 2$.
Consequently,  the indices $k$ and $s$ and element $y$ in the hypothesis of the following claim always exist. Moreover, if $k\in I_{m+2}\cap I_Z$, then we get the improved estimate $|A_k\cap Z|= |A_k|\geq m+2$, in which case the above argument yields  at least \emph{two} elements $y,\, y'\in A_k$ satisfying the hypotheses of Claim A.

\subsection*{Claim A} $A_s\nsubseteq A_k$ and  $X+\underset{i\neq k}{\Sum{i=1}{n}}A_i+(A_k\setminus \{y\})\neq X+\Sum{i=1}{n}A_i$ for any $k\in I_{m+2}$, $s\in I_m$  and $y\in (A_k\setminus A_s)\cap Z$ with $|(y+H)\cap A_k|\geq 2$.

\begin{proof}
If $X+\underset{i\neq k}{\Sum{i=1}{n}}A_i+(A_k\setminus \{y\})= X+\Sum{i=1}{n}A_i$, then we can remove  $y$ from $A_k$ and place it in $A_s$ to yield a new setpartition $\mathscr B=B_1\bdot\ldots\bdot B_n$, where $B_k=A_k\setminus \{y\}$, $B_s=A_s\cup \{y\}$, and $B_i=A_i$ for all $i\neq k,s$, such that $\mathsf S(\mathscr B)=S$, $X+\Sum{i=1}{n}A_i\subseteq X+\Sum{i=1}{n}B_i$, $Z\subseteq \bigcap_{i=1}^n(B_i+H)$ (since $|(y+H)\cap A_k|\geq 2$), and $|B_i\setminus Z|=|A_i\setminus Z|\leq 1$ for all $i$ (since $y\in Z$).
Thus $\mathscr B$ satisfies \eqref{lupper} and \eqref{brunch}. However,
since $|A_k|\geq |A_s|+2$, it follows that $|A_k|^2+|A_s|^2>|B_k|^2+|B_s|^2$, so that
$\mathscr B$ contradicts the minimality of \eqref{min-squares-gen} for $\mathscr A$. Therefore  we instead conclude that $X+\underset{i\neq k}{\Sum{i=1}{n}}A_i+(A_k\setminus \{y\})\neq X+\Sum{i=1}{n}A_i$.
If $A_s\subseteq A_k$, then let $y\in (A_k\setminus A_s)\cap Z$ be any element with $|(y+H)\cap A_k|\geq 2$, which exists as argued above Claim A. Then $A_s\subseteq A_k\setminus\{y\}$, whence $A_s+A_k\subseteq (A_s\cup\{y\})+(A_k\setminus \{y\})$, and removing $y$ from $A_k$ and placing it in $A_s$ yields a  new setpartition $\mathscr B$ that contradicts the minimality of \eqref{min-squares-gen} for $\mathscr A$ as before. Therefore $A_s\nsubseteq A_k$, completing the claim.
\end{proof}

Let $L=\mathsf H(X+\Summ{i\in I_m\cup I_{m+1}}A_i)$ and re-index the $A_i$ such that $J_Z=(I_m\cup I_{m+1})\cap I_Z=[1,n_Z]$, \ $J_e=(I_m\cup I_{m+1})\cap I_e=[n_Z+1,n_e]$ and  $I_{m+2}=[n_e+1,n]$.

\subsection*{Claim B} $|A_k+L|\geq |A_k|+3(|L|-1)$ for any $k\in I_{m+2}$.

\begin{proof}
Let $k\in I_{m+2}$ and $s\in I_m$.
Since $X+\underset{i\neq k}{\Sum{i=1}{n}}A_i$ is $L$-periodic (as $k\notin I_m\cup I_{m+1}$), it follows from Claim A that any element $y\in A_k\cap Z$ satisfying the hypotheses of Claim A must be the unique element from its $L$-coset in $A_k$. Since $L\leq H$, the same is true of any element $y\in A_k$ which is the unique element from its $H$-coset in $A_k$.
Since there is always at least one element satisfying the hypotheses of Claim A,  at least two when $k\in I_Z$, and also an  element from $A_k\setminus Z$ which is  the unique element from its $H$-coset in $A_k$ when $k\in I_e$,  the claim follows if there is any $y\in A_k\cap Z$ that is the unique element from its $H$-coset in $A_k$, so we instead assume $|(y+H)\cap A_k|\geq 2$ for all $y\in A_k\cap Z$. Thus $$|(A_k\cap Z)\setminus A_s|\geq |A_k\cap Z|-|A_s|+1=|A_k\cap Z|-m+1,$$ is the number of elements $y\in A_k\cap Z$ satisfying the hypothesis of Claim A, with the inequality following since Claim A ensures that $A_s\nsubseteq A_k$. If $k\in I_Z$, we obtain at least $|(A_k\cap Z)\setminus A_s|\geq |A_k\cap Z|-m+1= |A_k|-m+1\geq 3$ elements satisfying the hypotheses of Claim A. If $j\in I_e$, we obtain   at least $|(A_k\cap Z)\setminus A_s|\geq |A_k\cap Z|-m+1\geq |A_k|-m\geq 2$ elements satisfying the hypotheses of Claim A, as well as the element from $A_k\setminus Z$, which is the unique element from its $H$-coset in $A_k$. In either case, the claim follows.
\end{proof}

\subsection*{Claim C} Either
$|X+\Summ{i\in J_Z}A_i|\leq |X|+\Summ{i\in J_Z}|A_i|-|J_Z|-\max\{|L|-1,\,1\}$, or else $L$ is trivial and $|X+\Sum{i=1}{j}A_i|= |X+\Sum{i=1}{j-1}A_i|+|A_j|-1$ for all $j\in [1,n]$.

\begin{proof}
  Suppose \be\label{wally-gen}|X+\Summ{i\in J_Z}A_i|\geq|X|+\Summ{i\in J_Z}|A_i|-|J_Z|-\max\{|L|-1,\,1\}+1.\ee
 Note  $J_Z\cup J_e=I_m\cup I_{m+1}$ and $L=\mathsf H(X+\Summ{i\in J_Z}A_i+\Summ{i\in J_e}A_i)$.  Kneser' Theorem implies
\be\label{nelly-gen}|X+\Summ{i\in I_m\cup I_{m+1}}A_i|\geq |X+\Summ{i\in J_Z}A_i|+\Summ{i\in J_e}|A_i+L|-|J_e||L|.\ee

By definition, each $A_i$ with  $i\in J_e\subseteq I_e$ has some element $z\in A_i\setminus Z$ which is the unique element from its $H$-coset in $A_i$, meaning $(z+H)\cap A_i=\{z\}$. Since $L\leq H$, it is also the unique element from its $L$-coset in $A_i$, ensuring that $|A_i+L|\geq |A_i|+|L|-1$ for $i\in J_e$.  Combining this observation with \eqref{wally-gen} and \eqref{nelly-gen}, we conclude that
\be\label{silly-gen} |X+\Summ{i\in I_m\cup I_{m+1}}A_i|\geq |X|+\Summ{i\in I_m\cup I_{m+1}}|A_i|-|I_m\cup I_{m+1}|-\max\{|L|-1,\,1\}+1.\ee

 Let $k\in I_{m+2}=[n_e+1,n]$.  Since  $X+\Sum{i=1}{n_e}A_i$ is $L$-periodic, it follows that $X+\Sum{i=1}{k-1}A_i$ is also $L$-periodic. Consequently, in view of Claim A, it follows that there is a unique expression element in the sumset $\phi_{L}\Big(X+\Sum{i=1}{k-1}A_i\Big)+\phi_{L}(A_k)$. Thus Theorem \ref{thm-multbound} implies that
$|\phi_{L}\Big(X+\Sum{i=1}{k-1}A_i\Big)+\phi_{L}(A_k)|\geq |\phi_{L}\Big(X+\Sum{i=1}{k-1}A_i\Big)|+|\phi_{L}(A_k)|-1$ for $k\in I_{m+2}$. Lemma \ref{lem-local-cdtbound} now  yields
\be\label{sally-gen} |X+\Sum{i=1}{n}A_i|\geq |X+\Summ{i\in I_m\cup I_{m+1}}A_i|+\Summ{i\in I_{m+2}}|A_i+L|-|I_{m+2}|-(|L|-1)|I_{m+2}|.\ee
By Claim B, we have  $|A_k+L|\geq |A_k|+3(|L|-1)$ for all $k\in I_{m+2}$.
 Combining this observation with \eqref{silly-gen} and \eqref{sally-gen}, we deduce that
\be\label{tangol}|X+\Sum{i=1}{n}A_i|\geq |X|+\Sum{i=1}{n}|A_i|-n+2|I_{m+2}|(|L|-1)-\max\{|L|-1,1\}+1.\ee

Suppose $|X+\Sum{i=1}{n}A_i|\leq |X|+\Sum{i=1}{n}|A_i|-n$. Then, since $|I_{m+2}|\geq 1$, we must have equality in \eqref{tangol} with $L$ trivial, and equality must hold in all estimates used to derive \eqref{tangol}. Since $L$ is trivial, Kneser's Theorem implies  $|X+\Sum{i=1}{j}A_i|\geq |X|+\Sum{i=1}{j}|A_i|-j$ for all $j\in I_m\cup I_{m+1}=[1,n_e]$. But now Lemma \ref{lem-local-cdtbound} ensures that equality must hold in all these estimates, as otherwise \eqref{silly-gen} holds strictly, implying \eqref{tangol} also holds strictly. We also have equality holding in the modulo $L$ estimates used to derive \eqref{sally-gen}, meaning  $|X+\Sum{i=1}{k}A_i|= |X|+\Sum{i=1}{k}|A_i|-k$ for all $k\geq n_e+1$, for otherwise \eqref{sally-gen} holds strictly, and thus \eqref{tangol} as well. The claim now follows. So it remains to consider the case when  $$|X+\Sum{i=1}{n}A_i|> |X|+\Sum{i=1}{n}|A_i|-n=|S|-n+|X|,$$ in which case $Z=H=G$.

In this case, since $n\in I_{m+2}$ by our choice of indexing, Claim A ensures that every element $y\in A_n\setminus A_s$  is part of a unique expression element in $\Big(X+\Sum{i=1}{n-1}A_i\Big) +A_n$, where $s\in I_m$. If some $y\in A_n\setminus A_s$ is part of exactly one unique expression element, then $|X+\Sum{i=1}{n-1}A_i+(A_n\setminus \{y\})|=|X+\Sum{i=1}{n}A_i|-1\geq |X|+\Sum{i=1}{n}|A_i|-n$. Thus removing $y$ from $A_n$ and placing it in $A_s$ yields a new setpartition contradicting the minimality of \eqref{min-squares-gen} for $\mathscr A$. Therefore we can assume each $y\in A_n\setminus A_s$ is part of at least two unique expression elements. Since $A_s\nsubseteq A_n$ by Claim A, we have  $|A_n\setminus  A_s|\geq 3$, so there are at least three such elements in $A_n$, and thus at least two elements in $A_n\setminus\{y\}$ that are each part of at least two unique expression elements in the sumset $\Big(X+\Sum{i=1}{n-1}A_i\Big)+(A_n\setminus \{y\})$, for any fixed $y\in A_n\setminus A_s$.
In consequence,  the Kemperman Structure Theorem \cite[Theorem 9.1]{Gbook} \cite[Proposition 2.2]{kst+quasi} implies  \be\label{kstplus}|\Big(X+\Sum{i=1}{n-1}A_i\Big)+(A_n\setminus \{y\})|\geq |X+\Sum{i=1}{n-1}A_i|+|A_n\setminus \{y\}|.\ee If $L$ is trivial or $|I_{m+2}|\geq 2$,
 then repeating the  arguments used to derive \eqref{sally-gen} and \eqref{tangol} for $A_1\bdot\ldots\bdot A_{n-1}$ rather than $A_1\bdot\ldots\bdot A_n$ yields
$|X+\Sum{i=1}{n-1}A_i|\geq |X|+\Sum{i=1}{n-1}|A_i|-(n-1)$. Combined with \eqref{kstplus},  it follows that $|X+\Sum{i=1}{n-1}A_i+(A_n\setminus \{y\})|\geq |X|+\Sum{i=1}{n}|A_i|-n$, and now removing $y$ from $A_n$ and placing it in $A_s$ yields a new setpartition contradicting the minimality of \eqref{min-squares-gen} for $\mathscr A$.
Therefore we now assume $L$ is nontrivial and $|I_{m+2}|=1$, meaning $I_m\cup I_{m+1}=[1,n-1]$. In this case, since $H=G$ ensures $n\in I_Z\cap I_{m+2}$ so that there is at least one element from $A_n\setminus\{y\}$ satisfying the hypothesis of Claim A, we can repeat the arguments used to derive \eqref{sally-gen} for $A_1\bdot\ldots\bdot A_{n-1}\bdot (A_n\setminus \{y\})$ rather than $A_1\bdot\ldots\bdot A_n$  to find
\be\label{finding}|X+\Sum{i=1}{n-1}A_i+(A_n\setminus\{y\})|\geq |X+\Sum{i=1}{n-1}A_i|+|(A_n\setminus \{y\})+L|-|L|\geq |X+\Summ{i\in I_m\cup I_{m+1}}A_i|+|A_n|+|L|-3,\ee with the latter inequality as  Claim B ensures $|(A_n\setminus \{y\})+L|-|A_n\setminus \{y\}|\geq 2(|L|-1)$. Combining \eqref{finding} with \eqref{silly-gen} and using that $L$ is nontrivial, we obtain $|X+\Sum{i=1}{n-1}A_i+(A_n\setminus \{y\})|\geq |X|+\Sum{i=1}{n}|A_i|-n$. Once again, removing $y$ from $A_n$ and placing it in $A_s$ yields  a new setpartition contradicting the minimality of \eqref{min-squares-gen} for $\mathscr A$, completing  Claim C.
\end{proof}

We now split the proof into two cases depending on which outcome holds in Claim C.

\subsection*{CASE 1} $|X+\Summ{i\in J_Z}A_i|\leq |X|+\Summ{i\in J_Z}|A_i|-|J_Z|-\max\{|L|-1,\,1\}$ holds for every setpartition $\mathscr A=A_1\bdot\ldots\bdot A_n$ with $\mathsf S(\mathscr A)=S$ satisfying \eqref{lupper}, \eqref{brunch} and \eqref{min-squares-gen}.

\smallskip

For $j\in [0,n_Z]$, let $K_j=\mathsf H(X+\Sum{i=1}{j}A_i)\leq L$.
In view of the case hypothesis, let $j\in J_Z=[1,n_Z]$ be the minimal index such that $|X+\Sum{i=1}{j}A_i|\leq
|X|+\Sum{i=1}{j}|A_i|-j-\max\{|K_j|-1,\,1\}$.

\subsection*{Claim D}  $K_j=\mathsf H(X+\Sum{i=1}{j}A_i)=\mathsf H(X+\Sum{i=1}{j-1}A_i)$ and $|X+\Sum{i=1}{j}A_i|<|X+\Sum{i=1}{j-1}A_i|
+|A_j|-1$.

\begin{proof}  If $j=1$, then $|X+A_1|\leq |X|+|A_1|-2$, so that Kneser's Theorem implies $K_1$ is nontrivial with $|X|+|A_1|-|K_1|\geq |X+A_1|\geq |X+K_1|+|A_1|-|K_1|$ (both upper bounds for $|X+A_1|$ follow from the definition of $j$). Thus $X+K_1=X$, ensuring that $K_0=K_1$, and the claim follows. Therefore we now assume $j\in [2,n_Z]$. Hence the minimality of $j$ ensures
\be\label{doku1}|X+\Sum{i=1}{j-1}A_i|\geq |X|+\Sum{i=1}{j-1}|A_i|-(j-1)-\max\{|K_{j-1}|-1,\,1\}+1.\ee
Applying Kneser's Theorem again, we find \be\label{doku2}|X+\Sum{i=1}{j}A_i|\geq |X+\Sum{i=1}{j-1}A_i+K_j|+|A_j|-|K_j|.\ee
If $K_{j-1}\neq K_j$, then $K_j$ is nontrivial and
$|X+\Sum{i=1}{j-1}A_i+K_j|\geq |X+\Sum{i=1}{j-1}A_i|+|K_{j-1}|$, which combined with \eqref{doku1} and \eqref{doku2} yields
$|X+\Sum{i=1}{j}A_i|\geq
|X|+\Sum{i=1}{j}|A_i|-j+1-(|K_j|-1)=|X|+\Sum{i=1}{j}|A_i|-j+1-\max\{|K_j|-1,\,1\}$,
contrary to the definition of   $j$. Therefore $K_{j-1}=K_j$. If $|X+\Sum{i=1}{j}A_i|\geq|X+\Sum{i=1}{j-1}A_i|
+|A_j|-1$, then \eqref{doku1} yields $|X+\Sum{i=1}{j}A_i|\geq |X|+\Sum{i=1}{j}|A_i|-j+1-\max\{|K_{j-1}|-1,\,1\}=|X|+\Sum{i=1}{j}|A_i|-j+1-\max\{|K_{j}|-1,\,1\}$, again contrary to the definition of $j$. Therefore $|X+\Sum{i=1}{j}A_i|<|X+\Sum{i=1}{j-1}A_i|
+|A_j|-1$, and the claim follows.
\end{proof}

All the above is valid for any $\mathscr A$ with $\mathsf S(\mathscr A)=S$ which satisfies \eqref{lupper}, \eqref{brunch} and  \eqref{min-squares-gen}.
We  now impose additional extremal conditions on $\mathscr A$:
\begin{itemize}
\item[(a)] $|J_Z|$ is  minimal (subject to satisfying \eqref{lupper}, \eqref{brunch} and \eqref{min-squares-gen}), say $|J_Z|=n_Z$, with the $A_i$ indexed so that $J_Z=[1,n_Z]$.
\item[(b)] $j$ is maximal (subject to satisfying \eqref{lupper}, \eqref{brunch}, \eqref{min-squares-gen} and (a)).
\item[(c)] $\Sum{i=1}{j}|\phi_K(A_i)|$ is maximal, where $K=\mathsf H(X+\Sum{i=1}{j-1}A_i)$ (subject to satisfying \eqref{lupper}, \eqref{brunch}, \eqref{min-squares-gen},  (a) and (b)).
\end{itemize}

Let $k\in I_{m+2}$ be fixed and set  $K:=K_j\leq L\leq H$, so $$K=\mathsf H(X+\Sum{i=1}{j-1}A_i)=\mathsf H(X+\Sum{i=1}{j}A_i)$$ by Claim D. Since Claim D ensures $|(X+\Sum{i=1}{j-1}A_i)+A_j|<|X+\Sum{i=1}{j-1}A_i|+|A_j|-1$, it follows from Kneser's Theorem applied to $(X+\Sum{i=1}{j-1}A_i)+A_j$ that $K$ is nontrivial and  \be\label{K-holes}|A_j+K|-|A_j|\leq |K|-2.\ee In particular,
\be\label{K-full} |(x+K)\cap A_j|\geq 2\quad\mbox{ for all $x\in A_j$}.\ee
Since $j\in J_Z\subseteq I_Z$, we have $A_j\subseteq Z$. Let $Z_0=(A_j+K)\setminus  (A_k+K)\subseteq Z$. Let $X_0\subseteq Z_0\cap A_j$ be a subset consisting of one element from $A_j$ for every $K$-coset contained in $Z_0$. Thus $$|Z_0|=|K|\,|X_0|$$ with $Z_0$ the union of all $K$-cosets that intersect $A_j$ but not $A_k$.

\subsection*{Claim E}  There is a subset $Y\subseteq(A_k\cap Z)\setminus (A_j+K)$ with $|Y|\geq \max\{1,|X_0|\}$ and $(A_k\setminus Y)+H=A_k+H$.

\begin{proof} Let us first show there is some  $y\in (A_k\cap Z)\setminus (A_j+K)$ with $|(y+H)\cap A_k|\geq 2$.
Let $y_0\in A_k\cap Z$ be an element satisfying the hypotheses of Claim A.
Then $|(y_0+H)\cap A_k|\geq 2$.
Moreover, since $X+\Summ{i\in I_m\cup I_{m+1}}A_i$ is $K$-periodic with $k\in I_{m+2}$, the conclusion of Claim A ensures that  $(y_0+K)\cap A_k=\{y_0\}$.
If $y_0\notin A_j+K$, then taking $y=y_0$ yields the desired element $y$.
Therefore we may assume $y_0\in A_j+K$.
Then $Z_0\cup (y_0+K)\subseteq A_j+K$ with $\Big(Z_0\cup (y_0+K)\Big)\cap A_k=\{y_0\}$.
Thus $|A_k\cap (A_j+K)|\leq |A_j+K|-|Z_0|-|K|+1\leq |A_j|-|Z_0|-1$, with the second inequality in view of \eqref{K-holes}. It follows that \be\label{dracule}|(A_k\cap Z)\setminus (A_j+K)|= |A_k\cap Z|-|A_k\cap (A_j+K)|\geq (|A_k|-1)-(|A_j|-|Z_0|-1)= |A_k|-|A_j|+|Z_0|.\ee
In particular, since $|A_k|-|A_j|\geq (m+2)-(m+1)=1$ (as $k\in I_{m+2}$ and $j\in I_m\cup I_{m+1}$), we conclude that $(A_k\cap Z)\setminus (A_j+K)$ is nonempty.
If there is some $y\in (A_k\cap Z)\setminus (A_j+K)$ with $|(y+H)\cap A_k|\geq 2$, then the desired element $y$ is found.
Otherwise, we conclude that each $y\in (A_k\cap Z)\setminus (A_j+K)$ is the unique element from its $H$-coset in $A_k$.
However, since $y\in Z\subseteq A_j+H$, it follows that $(y+H)\cap A_j$ is also nonempty, say with $y'\in (y+H)\cap A_j$.
Since $y\notin A_j+K$, we have $y\notin y'+K$, ensuring that $y'+K\neq y+K$. Thus, as $(y+H)\cap A_k=\{y\}$, we conclude that $(y'+K)\cap A_k$ is empty, meaning $y'+K\subseteq  Z_0$. As this is true for each $y\in (A_k\cap Z)\setminus (A_j+K)$, with the corresponding sets $y'+K$ each lying in the distinct cosets $y+H$ for  $y\in (A_k\cap Z)\setminus (A_j+K)$ (as each such $y$ is the unique element from its $H$-coset in $A_k$), it follows that $$|Z_0|\geq |(A_k\cap Z)\setminus (A_j+K)|\, |K|\geq |(A_k\cap Z)\setminus (A_j+K)|.$$ However, applying this estimate in \eqref{dracule} yields the contradiction $m+2\leq |A_k|\leq |A_j|\leq m+1$. Thus the existence of the desired element  $y\in (A_k\cap Z)\setminus (A_j+K)$ with $|(y+H)\cap A_k|\geq 2$ is established, and we  assume $|X_0|\geq 2$ as the claim is now complete taking $Y=\{y\}$ when $|X_0|\leq 1$.

By definition,  $Z_0\cap A_k=\emptyset$ and $Z_0\subseteq A_j+K$. Thus  $|A_k\cap (A_j+K)|\leq |A_j+K|-|Z_0|\leq |A_j|-|Z_0|+|K|-2=|A_j|-|K|(|X_0|-1)-2$, with the second inequality in view of \eqref{K-holes}. It follows that \be\label{dracule2}|(A_k\cap Z)\setminus (A_j+K)|= |A_k\cap Z|-|A_k\cap (A_j+K)|\geq  |A_k|-|A_j|+|K|(|X_0|-1)+1.\ee
 Let  $Z'_0=\Big((A_k\cap Z)\setminus (A_j+K)+H\Big)\cap (Z_0+H)$ and partition $\Big((A_k\cap Z)\setminus (A_j+K)+H\Big)=Z'_0\cup Z_1$.
Since each $H$-coset in $Z'_0$ contains a $K$-coset from $Z_0$, we have \be\label{xogo}|X_0|\geq |Z'_0|/|H|.\ee
Let $Y\subseteq (A_k\cap Z)\setminus (A_j+K)$ be obtained by taking the set $(A_k\cap Z)\setminus (A_j+K)$ and  removing one element from $(A_k\cap Z)\setminus (A_j+K)$ from each of the $|Z'_0|/|H|$ $H$-cosets contained in $Z'_0$.
Then \be\label{tack} |Y|=|(A_k\cap Z)\setminus (A_j+K)|-|Z'_0|/|H|\geq |(A_k\cap Z)\setminus (A_j+K)|-|X_0|,\ee with the inequality in view of \eqref{xogo}, and $(\alpha+H)\cap (A_k\setminus Y)$ is nonempty for every $\alpha+H\subseteq Z'_0$ (as one element for each of these $H$-cosets was left out of $Y$, and thus remains in $A_k\setminus Y$).
  For each $\alpha+H\subseteq Z_1\subseteq Z$, we have $\alpha+H\subseteq Z\subseteq  A_j+H$, and thus there is some $\alpha'\in \alpha+H$ with $(\alpha'+K)\cap A_j$ nonempty.
  Since $\alpha+H\nsubseteq Z'_0$, it follows by definition of $Z'_0$ and $Z_0$ that $(\alpha'+K)\cap A_k$ is nonempty, and necessarily disjoint from $(A_k\cap Z)\setminus (A_j+K)$, and thus also from $Y\subseteq (A_k\cap Z)\setminus (A_j+K)$. In consequence, we have $(A_k\setminus Y)+H=A_k+H$. If $|Y|\geq |X_0|$, the claim is complete, so we instead assume $|Y|\leq |X_0|-1$,
  in which case
  \eqref{tack} implies $|(A_k\cap Z)\setminus (A_j+K)|\leq 2|X_0|-1$. However, using this estimate in \eqref{dracule2} along with $|A_k|-|A_j|\geq (m+2)-(m+1)=1$ yields
  $(|K|-2)(|X_0|-1)+1\leq 0$, which is not possible since  $K$ is nontrivial (as noted above \eqref{K-holes}) and $|X_0|\geq 1$, completing the claim.
\end{proof}

In view of Claim E, there is a nonempty subset $Y\subseteq (A_k\cap Z)\setminus (A_j+K)$ with $(A_k\setminus Y)+H=A_k+H$ and $|Y|=\max\{1,\,|X_0|\}$.
Define a new setpartition $\mathscr B=B_1\bdot\ldots\bdot B_n$ by setting $B_j=(A_j\setminus X_0)\cup Y$, $B_k=(A_k\setminus Y)\cup X_0$ and $B_i=A_i$ for all $i\neq k,j$. Since $K=\mathsf H(X+\Sum{i=1}{j-1}A_i)$, \eqref{K-full} ensures that $\phi_K(A_j)=\phi_K(A_j\setminus X_0)$ and $X+\Sum{i=1}{j}A_i=X+\Sum{i=1}{j-1}A_i+(A_j\setminus X_0)$.  As a result,
\be\label{periodpeice}X+\Sum{i=1}{n}A_i\subseteq X+\underset{i\neq j,k}{\Sum{i=1}{n}}A_i+(A_j\setminus X_0)+(A_k\cup X_0).\ee
By definition of $Y$, we have $\phi_K(Y)$ disjoint from $\phi_K(A_j)$, and thus also from $\phi_K(X_0)$, while the definition of $X_0$ ensures that $\phi_K(X_0)$ is disjoint from $\phi_K(A_k)$ with  $\phi_K(A_j\setminus X_0)=\phi_K(A_j)\subseteq \phi_K(A_k\cup X_0)\setminus \phi_K(Y)$. It follows that \begin{align*}\phi_K(A_j\setminus X_0)+\phi_K(A_k\cup X_0)&\subseteq \Big(\phi_K(A_j\setminus X_0)\cup \phi_K(Y)\Big)+\Big(\phi_K(A_k\cup X_0)\setminus \phi_K(Y)\Big)\\\subseteq \phi_K(B_j)+\phi_K(B_k).\end{align*}
As a result, since $X+\Sum{i=1}{j-1}A_i=X+\Sum{i=1}{j-1}B_i$ is $K$-periodic, we conclude from \eqref{periodpeice} that $$X+\Sum{i=1}{n}A_i\subseteq X+\underset{i\neq j,k}{\Sum{i=1}{n}}A_i+(A_j\setminus X_0)+(A_k\cup X_0)\subseteq X+\Sum{i=1}{n}B_i,$$ so \eqref{lupper} holds for $\mathscr B$.
Since $X_0\subseteq A_j\subseteq Z$ (as $j\in J_Z$), Claim E ensures that $Z\subseteq A_k+H=B_k+H$ and $|B_k\setminus Z|=|A_k\setminus Z|\leq 1$. Since $\phi_K(A_j)\subseteq \phi_K(B_j)$ with $Y\subseteq Z$ by definition, we have $Z\subseteq B_j+H=A_j+H$ and $|B_j\setminus Z|=|A_j\setminus Z|=0$. Consequently, $Z\subseteq \bigcap_{i=1}^{n}(B_i+H)$ and $|B_i\setminus Z|\leq 1$ for all $i$, in which case $\mathscr B$ satisfies  \eqref{brunch}. We have  $|B_j|=|A_j|-|X_0|+|Y|$, $|B_k|=|A_j|-|Y|+|X_0|$, and $|B_i|=|A_i|$ for $i\neq j,k$.
Let $I'_m$, $I'_{m+1}$, $I'_{m+2}$, $I'_e$, $I'_Z$ and  $J'_Z$  be the associated quantities $I_m$, $I_{m+1}$, $I_{m+2}$, $I_e$, $I_Z$ and $J_{Z}$  for $\mathscr B$ rather than $\mathscr A$.

Suppose $|X_0|=0$.  Then $|Y|=1$, \  $|B_j|=|A_j|+1$ and $|B_k|=|A_k|-1$. If $|A_k|\geq |A_j|+2$, then $\mathscr B$ contradicts the minimality of \eqref{min-squares-gen} for $\mathscr A$.
Otherwise, we have $|A_k|=|A_j|+1=m+2$, $|B_j|=|A_k|=m+2$ and $|B_k|=|A_j|=m+1$ so that $\Sum{i=1}{n}|B_i|^2=\Sum{i=1}{n}|A_i|^2$, meaning $\mathscr B$ satisfies the extremal condition \eqref{min-squares-gen}.
Now $I'_{m+1}=I_{m+1}\setminus \{j\}\cup \{k\}$, \ $I'_m=I_m$ and $I'_{m+2}=I_{m+2}\setminus\{k\}\cup \{j\}$. If $k\in I_e$, then $k\in I'_e$ (as $Y\subseteq Z$), in which case $J'_Z=J_Z\setminus\{j\}$, contradicting the minimality condition (a) for $\mathscr A$. On the other hand, if $k\in I_Z$, then $J'_Z=J_Z\setminus\{j\}\cup \{k\}$, so condition (a) holds for $\mathscr B$.
Swapping the indices on $B_k$ and $B_j$, so now $B_k=(A_j\setminus X_0)\cup Y$ and  $B_j=(A_k\setminus Y)\cup X_0$, we obtain $J'_Z=J_Z$ and  $I'_{m+1}\cup I'_m=I_{m+1}\cup I_m$. Since $A_i=B_i$ for $i<j$, the definition of $j$ ensures $j'\geq j$, where $j'$ is the associated quantity for $\mathscr B$ corresponding to the index $j$ for $\mathscr A$,  while the extremal condition given in (b) forces $j'\leq j$. Thus $j'=j$. However, since $k\in I_Z$, there are at least two elements $y$ satisfying the hypotheses of Claim A for $\mathscr A$, which in view of the conclusion of Claim A and $K=\mathsf H(X+\Sum{i=1}{j-1}A_i)=\mathsf H(X+\Sum{i=1}{j'-1}B_i)$, means both these elements are the unique element  from their $K$-coset in $A_k$.  As at most one of them can be contained in the singleton set $Y$, we conclude that $B_j=A_k\setminus Y$ contains some $y\in B_j$ with $|(y+K)\cap B_j|=1$. However, in such case, \eqref{K-full} could not hold for the index $j$ in $\mathscr B$, contradicting that it must hold for $j'=j$ by the arguments above. So we instead conclude that $|X_0|\geq 1$,

Since $|X_0|\geq 1$, we have  $|X_0|=|Y|$, $|B_j|=|A_j|$ and $|B_k|=|A_k|$, so \eqref{min-squares-gen} holds for $\mathscr B$ with $I'_m=I_m$, $I'_{m+1}=I_{m+1}$ and $I'_{m+2}=I_{m+2}$. Since $Y\subseteq Z$ and $A_j\subseteq Z$ (as $j\in J_Z$), we conclude that $J'_Z=J_Z$, meaning condition (a) holds for $\mathscr B$. As argued in the previous case, we must have $j'=j$, so that condition (b) holds. In particular, we must have $K=\mathsf H(X+\Sum{i=1}{j-1}A_i)=\mathsf H(X+\Sum{i=1}{j'-1}B_i)$.
By definition of $Y$, each $y\in Y$ is disjoint from $A_j+K$. Thus, since $|Y|=|X_0|\geq 1$ and $\phi_K(A_j\setminus X_0)=\phi_K(A_j)$, we conclude that $|\phi_K(B_j)|>|\phi_K(A_j)|$, in  which case the maximality condition (c) for $\mathscr A$ is contradicted by $\mathscr B$, completing CASE 1.

\subsection*{CASE 2}  $L$ is trivial and $|X+\Sum{i=1}{j}A_i|= |X+\Sum{i=1}{j-1}A_i|+|A_j|-1$ for all $j\in [1,n]$, for some setpartition $\mathscr A=A_1\bdot\ldots\bdot A_n$ with $\mathsf S(\mathscr A)=S$ satisfying \eqref{lupper}, \eqref{brunch} and \eqref{min-squares-gen}, where $J_Z=[1,n_Z]$, \ $J_e=[n_Z+1,n_e]$ and $I_{m+2}=[n_e+1,n]$.

\smallskip

Let $Y=X+\Sum{i=1}{n-1}A_i$ and $V=X+\Sum{i=1}{n}A_i=Y+A_n$. In view of the case hypothesis and Lemma \ref{lem-local-cdtbound}, we have  $|X+\Sum{i=1}{j}A_i|= |X|+\Sum{i=1}{j}|A_i|-j$ for all $j\in [1,n]$. In particular, $|X+\Summ{i\in J_Z}A_i|=|X|+\Summ{i\in J_Z}|A_i|-|J_Z|$ and \be\label{large} |V|=|X+\Sum{i=1}{n}A_i|= |X|+\Sum{i=1}{n}|A_i|-n=|S|-n+|X|.\ee The former equality combined with Claim C ensures that  the hypotheses of CASE 2 hold for $\mathscr A$ under any re-indexing of the $A_i$ with $i\in I_{m+2}=[n_e+1,n]$, allowing us to freely assume an arbitrary set $A_k$ with  $k\in I_{m+2}$ occurs with $k=n$.
We aim to either contradict the extremal condition \eqref{min-squares-gen} or show that Item 1 holds.

\subsection*{Claim F} If $H=\mathsf H(X+\Sum{i=1}{n}A_i)$, then $A_k\subseteq Z$ with $|(y+H)\cap A_k|\geq 2$ for all $k\in I_{m+2}$ and $y\in A_k$.

\begin{proof}
Suppose $H=\mathsf H(X+\Sum{i=1}{n}A_i)$. Note $H$ is nontrivial as remarked after \eqref{min-squares-gen}.  By our choice of indexing, $n\in I_{m+2}$. Let $s\in I_m$. Each $y\in A_n$ satisfying the hypothesis of Claim A is a unique expression element in $Y+A_n$, of which there is at least one. Thus, since $Y+A_n$ is $H$-periodic with $|Y+A_n|=|Y|+|A_n|-1$ by case hypothesis, we can apply the Kemperman Structure Theorem directly to $Y+A_n$ to conclude that there are $H$-quasi-periodic decompositions (cf. \cite[Comment c.14]{kst+quasi}) $Y=Y_1\cup Y_0$ and $A_n=X_1\cup X_0$ with $|Y_0|+|X_0|=|H|+1$. Moreover, in view of \cite[Theorem 5.1]{Gbook} and $H$ nontrivial, either all unique expression elements are contained in $Y_0+X_0$, or $|X_0|=1$ with all unique expression elements involving the unique element in $X_0$, or $|Y_0|=1$ with all unique expression elements involving the unique element in $Y_0$. If $|X_0|=1$, then $|Y_0|+|X_0|=|H|+1$ ensures $|Y_0|=|H|\geq 2$, in which case all unique expression elements in $Y+A_n$ must involve the unique element from $X_0$. However, this contradicts that there is an  element $y\in A_n$ satisfying Claim A, which is part of a unique expression element in $Y+A_n$ but not the unique element from its $H$-coset. Therefore $|X_0|\geq 2$, ensuring that $|(y+H)\cap A_n|\geq 2$ for all $y\in A_n$. Since any element from $A_n\setminus Z$ is the unique element from its $H$-coset in $A_n$, it follows that $A_n\subseteq Z$. Repeating the above argument for an arbitrary $A_k$ with $k\in I_{m+2}$ (using an appropriate re-indexing), we conclude that $|(y+H)\cap A_k|\geq 2$ for all $y\in A_k$, and that  $A_k\subseteq Z$, which completes the claim.
\end{proof}

Let $s\in I_m$ be arbitrary. Recall  $n\in I_{m+2}$ by our choice of indexing. In view of Claim F, any element $y\in A_n\setminus A_s$ satisfies the hypotheses of Claim A (this is trivially true if $H=G$). Thus, since Claim A ensures that $A_s\nsubseteq A_n$, we conclude that there are $|A_n\setminus A_s|\geq 3$ elements satisfying the hypotheses of Claim A. Each such $y\in A_n\setminus A_s$ is part of a unique expression element in $Y+A_n$ by Claim A. As there are at least three such $y$, and since $|Y+A_n|=|Y|+|A_n|-1$ by case hypothesis,
 the Kemperman Structure Theorem \cite[Theorem 9.1]{Gbook} \cite[Proposition 2.2]{kst+quasi} ensures this is only possible if there are $K$-quasi-periodic decompositions $Y=(Y\setminus\{y\})\cup \{y\}$ and $A_n=(A_n\setminus A_\emptyset)\cup A_\emptyset$ with $y+A_\emptyset\subseteq Y+A_n$ the subset of all unique expression elements in $Y+A_n$, and $$K=\la A_\emptyset-A_\emptyset\ra.$$ In particular, $A_n\setminus A_s\subseteq A_\emptyset$ with each $x\in A_\emptyset$ being part of exactly one unique expression element $y+x\in Y+A_n$. Moreover, since $A_\emptyset$ is a subset of a $K$-coset, we have $|K|\geq |A_\emptyset|\geq |A_n\setminus A_s|\geq 3$. Consequently, in view of the case hypothesis and Lemma \ref{lemma-kst-apuncunctured}.3, it follows by a short inductive argument that there are $a_i\in A_i$ for $i\in [1,n-1]$ and $\beta\in X$ such that $X\setminus \{\beta\}$ and $A_i\setminus \{a_i\}$ for $i\in [1,n-1]$ are $K$-periodic with $\beta+a_1+\ldots+a_{n-1}=y$.

Let $x_1,\ldots,x_r\in A_n\setminus A_s\subseteq A_\emptyset$ be the $r\geq 3$ distinct elements in $A_n\setminus A_s$.
By translating all terms of $S$ appropriately, we can w.l.o.g. assume $a_s=0$. Consider an arbitrary  $x_t\in A_n\setminus A_s\subseteq A_\emptyset$. By the above work, $Y+(A_n\setminus \{x_t\})=V\setminus \{y+x_t\}$.  If $V\setminus \{y+x_t\}$ is aperiodic, then Lemma \ref{lem-kt} together with Kneser's Theorem implies  that $|X+(A_s\cup \{x_t\})+\underset{i\neq s}{\Sum{i=1}{n-1}}A_i+(A_n\setminus\{x_t\})|
\geq |A_s\cup \{x_t\}|+|X+\underset{i\neq s}{\Sum{i=1}{n-1}}A_i+(A_n\setminus \{x_t\})|-1\geq
|X|+\Sum{i=1}{n}|A_i|-n$,
in which case moving $x_t$ from $A_n$ to $A_s$ yields a new setpartition satisfying \eqref{lupper} and \eqref{brunch} (in view of Claim F),  thus  contradicting the minimality of \eqref{min-squares-gen} for $\mathscr A$. Therefore we instead conclude $H_t:=\mathsf H(V\setminus \{y+x_t\})$ is nontrivial for every $x_t\in A_n\setminus A_s\subseteq A_\emptyset$.
Since there are at least two such elements, \cite[Proposition 2.1]{kst+quasi} implies the $H_t$  are distinct cardinality two subgroups  for  $t\in A_n\setminus A_s$.
Moreover, $V\cup \{\alpha\}$ is $(H_1+\ldots+H_r)$-periodic for the unique element  \be\label{findingalpha}\alpha\in (y+x_t+H_t)\setminus \{y+x_t\}.\ee
Since $V\setminus \{y+x_t\}$ is periodic,  \cite[Comment c.6]{kst+quasi} implies \be\label{Vaperiodic} V=X+\Sum{i=1}{n}A_i\quad \mbox{ is aperiodic}.\ee
Thus, since $H$ is nontrivial (as noted after \eqref{min-squares-gen}), we conclude that $H\neq \mathsf H(X+\Sum{i=1}{n}A_i)$, leaving us in the situation where $H=Z=G$ with $I_e$ empty.


We must have $\phi_K(\beta)+\Sum{i=1}{n}\phi_K(a_i)\in \phi_K(X)+\Sum{i=1}{n}\phi_K(A_i)$ a unique expression element, where $a_n\in A_\emptyset$, for otherwise $V=X+\Sum{i=1}{n}A_i$ will be $K$-periodic (as $X\setminus \{\beta\}$, $A_n\setminus A_\emptyset$ and $A_i\setminus \{a_i\}$ for $i\in [1,n-1]$ are all $K$-periodic), contrary to \eqref{Vaperiodic}.
In particular,
$$V=X+\Sum{i=1}{n}A_i=Z_{[1,n]}\cup (\beta+\Sum{i=1}{n-1}a_i+A_\emptyset)=Z_{[1,n]}\cup (y+A_\emptyset)$$ for some
%
$K$-periodic set $Z_{[1,n]}\subseteq G$.
 Since $X+\Sum{i=1}{n-1}A_i+(A_n\setminus\{x_t\})=Z_{[1,n]}\cup (y+A_\emptyset\setminus \{x_t\})$ is a $K$-quasi-periodic decomposition with $H_t=\mathsf H\Big(Z_{[1,n]}\cup (y+A_\emptyset\setminus \{x_t\})\Big)$ and $y+A_\emptyset\setminus\{x_t\}$ a nonempty,  proper subset of a $K$-coset, we must also have (by \eqref{dumpling}) \be\label{Htsome}H_t=\mathsf H(A_\emptyset\setminus \{x_t\})\leq K\quad\mbox{ for any $x_t\in A_n\setminus A_s$}.\ee
As a result, since $X\setminus \{\beta\}$ and $A_i\setminus \{a_i\}$ are $K$-periodic, it follows that $|X+H_t|=|X|+1$ and $|A_i+H_t|=|A_i|+1$ for all $i\in [1,n-1]$. Moreover, $A_n\setminus \{x_t\}$ is $H_t$-periodic.  Now $0=a_s\in A_s$ with $A_s\setminus \{0\}$ $H_t$-periodic. Hence, if $\{x_t,a_s\}=\{x_t,0\}\neq H_t$, then $|(A_s\cup \{x_t\})+H_t|=|A_s|+3$. In such case,  Lemma \ref{lem-kt} together with Kneser's Theorem implies
$|X+(A_s\cup \{x_t\})+\underset{i\neq s}{\Sum{i=1}{n-1}}A_i+(A_n\setminus\{x_t\})|\geq |X+H_t|+|(A_s\cup \{x_t\})+H_t|+\underset{i\neq s}{\Sum{i=1}{n-1}}|A_i+H_t|+|(A_n\setminus\{x_t\})+H_t|-n|H_t|=|X|+\Sum{i=1}{n}|A_i|-n+1$, in which case moving $x_t$ from $A_n$ to $A_s$ yields a new setpartition  contradicting the minimality  of \eqref{min-squares-gen} for $\mathscr A$.
 Therefore we instead conclude that $$H_t=\{0,x_t\}\leq K, \quad\mbox{ for each $x_t\in A_n\setminus A_s$,}$$ is a cardinality two subgroup.
 Repeating the above arguments using any $i\in I_m$ in place of $s$, we find that  $a_i=0$ for all $i\in I_m$ (as $\{a_i, x_t\}$ must equal a single $H_t$-coset with $x_t\in H_t$ the unique nonzero element of $H_t$). Since $H_t=\{0,x_t\}$, it follows from \eqref{findingalpha} that $$\alpha=y+2x_t=y.$$ Thus $V\cup \{y\}=Z_{[1,n]}\cup \Big(y+(A_\emptyset\cup \{0\})\Big)$ is $(H_1+\ldots+H_r)$-periodic with $H_1+\ldots+H_r\leq K$, ensuring that $A_\emptyset\cup \{0\}$ is also $(H_1+\ldots+H_r)$-periodic.

Suppose $|A_n|\geq m+3$ and let $K'=H_1+H_2=\{0,x_1,x_2,x_1+x_2\}\leq K$. As just noted, $A_\emptyset \cup \{0\}$ is  $(H_1+\ldots+H_r)$-periodic, and thus also $K'$-periodic with $K'\leq K$. Consequently, $A_n\setminus \{x_1,x_2\}=Z'\cup \{x_1+x_2\}$ with $Z':=A_n\setminus K'$ a $K'$-periodic set and $x_1+x_2$ the unique element from its $K'$-coset in $A_n\setminus \{x_1,x_2\}$.
The sets $X\setminus \{\beta\}$ and $A_i\setminus\{a_i\}$ for $i\in [1,n-1]$ are all $K$-periodic, and thus also $K'$-periodic. The  set $\Big(A_n\setminus \{x_1,x_2\}\Big)\setminus \{x_1+x_2\}=Z'=A_n\setminus K'$ is also $K'$-periodic with $\phi_{K'}(x_1+x_2)=0$. It follows that
$\phi_{K'}(y)=\phi_{K'}(\beta)+\Sum{i=1}{n-1}\phi_{K'}(a_i)\in \phi_{K'}(X)+\Sum{i=1}{n}\phi_{K'}(A_i)$ must be a unique expression element, as otherwise $X+\Sum{i=1}{n}A_i$ would be $K'$-periodic, contradicting \eqref{Vaperiodic}, and thus $X+(A_s\cup\{x_1,x_2\})+\underset{i\neq s}{\Sum{i=1}{n-1}}A_i+(A_n\setminus \{x_1,x_2\})=V\setminus (y+K')\cup (y+\{0,x_1,x_2\}+\{x_1+x_2\})=V$
 Removing $x_1$ and $x_2$ from $A_n$ and placing them in $A_s$  now  yields a new setpartition with the same cardinality sumset as $\mathscr A$, contradicting the minimality of \eqref{min-squares-gen} for $\mathscr A$ (as $|A_n|\geq m+3$). So we conclude that $|A_n|=m+2$. Re-indexing the $A_k$ with $k\in I_{m+2}$ and repeating these arguments for any $A_k$ with $k\in I_{m+2}$, we conclude that $|A_k|=m+2$ for all $k\in I_{m+2}$.

Since $A_n\setminus A_s\subseteq A_\emptyset$, we have $A_n\setminus A_\emptyset\subseteq A_s$. Hence, since $A_n\setminus A_\emptyset$ is $K$-periodic and $0\in A_s$ is the unique element from it $K$-coset in $A_s$, it follows that $A_n\setminus A_\emptyset\subseteq A_s\setminus \{0\}$. If $A_n\setminus A_\emptyset\neq A_s\setminus \{0\}$, then $A_s\setminus \{0\}$  and $A_n\setminus A_\emptyset$ being $K$-periodic ensures $|A_s|\geq |A_n\setminus A_\emptyset|+|K|+1\geq |A_n\setminus A_\emptyset|+|A_\emptyset|+1\geq |A_n|+1$, which is not possible. We are left to conclude $A_n\setminus A_\emptyset=A_s\setminus \{0\}$.
Thus $m-1=|A_s\setminus \{0\}|=|A_n\setminus A_\emptyset|$, implying $m+2=|A_n|=(m-1)+|A_\emptyset|$ and  $|A_\emptyset|=3$, and since $A_n\setminus A_s\subseteq A_\emptyset$ is a set of size at least three, we conclude that $A_\emptyset=A_n\setminus A_s$ and $A_n\cap A_s=A_n\setminus A_\emptyset=A_s\setminus \{0\}$. In particular, $K=\la A_\emptyset -A_\emptyset\ra =H_1+H_2+H_3=\la x_1,x_2,x_3\ra$. Since $V\cup \{y\}=Z_{[1,n]}\cup \Big(y+(A_\emptyset\cup \{0\})\Big)$ is $K$-periodic, with $Z_{[1,n]}$ a $K$-periodic set, it follows that   $\{0\}\cup A_\emptyset=\{0,x_1,x_2,x_3\}=K$  is an elementary $2$-group of order $4$, whence $K\cong (\Z/2\Z)^2$.
Repeating these arguments for any $s\in I_m$ and $k\in I_{m+2}$, it follows that there exists a $K$-periodic subset $W\subseteq G\setminus K$ such that \be\label{equalsets}A_s=W\cup \{0\}\quad\und\quad A_k=W\cup (K\setminus \{0\})\quad\mbox{ for every $s\in I_m$ and $k\in I_{m+2}$}.\ee
Since $A_s\setminus \{0\}$ is $K$-periodic, we have $|A_s\setminus \{0\}|=m-1$ divisible by $|K|=4$. Any $j\in I_{m+1}$ also has $A_j\setminus \{a_j\}$ $K$-periodic, whence $m=|A_j|-1$ is divisible by $4$. Since $m-1$ and $m$ cannot both be divisible by $4$,  it follows  that $I_{m+1}$ is empty.

Suppose $|I_{m+2}|\geq 2$. 
Then \eqref{equalsets} implies $A_{n-1}=A_n$. Since $A_{n-1}\setminus\{a_{n-1}\}$ is $K$-periodic, we have $|A_{n-1}|\equiv 1\mod |K|$. Since $A_n\setminus A_\emptyset$ is $K$-periodic with $|A_\emptyset |=3$, we have $|A_n|\equiv 3\mod |K|$. However, since $|K|=4$, this contradicts that  $A_{n-1}=A_n$. So we conclude that $|I_{m+2}|=1$.

We now know  $A_1=A_2=\ldots=A_{n-1}=W\cup \{0\}$ and  $A_n=W\cup \{x_1,x_2,x_3\}$ with $W\subseteq G\setminus K$ and $X\setminus \{\beta\}$ $K$-periodic sets and $K=\{0,x_1,x_2,x_3\}$.
If $n\geq 3$, then  consider the setpartition $\mathscr A=A'_1\bdot\ldots\bdot A'_n$ with $A'_i=W\cup \{0\}$ for $i\in [1,n-3]$, $A'_{n-2}=W\cup \{x_1\}$, $A'_{n-1}=W\cup \{0,x_2\}$ and $A'_n=W\cup \{0,x_3\}$. Then $X+\Sum{i=1}{n}A'_i=V\cup \{\beta\}=V+K$, so that $\mathscr A'$ contradicts the minimality of \eqref{min-squares-gen} for $\mathscr A$ in view of \eqref{large}.  Therefore $n=2$ (as we assumed $n\geq 2$ at the very start of the proof). It is now readily checked that $\mathscr A=A_1\bdot A_2$ with $A_1=W\cup \{x\}$ and $A_2=W\cup (K\setminus\{x\})$, for $x\in K$, are the  only setpartitions partitioning the terms of $S$ with $|X+\Sum{i=1}{2}A_i|\geq \min\{|X|+\Sum{i=1}{2}|A_i|-2,\,|X+\Sigma_2(S)|\}=|X|+\Sum{i=1}{2}|A_i|-2=|V|$, so the original setpartition $\mathscr A$ from the hypotheses must have this form. As the above works shows Item 1 holds for such $\mathscr A$, the case and  proof  is complete. \end{proof}

We can now proceed with the proof of Theorem \ref{thm-main-ccd}.

\begin{proof}[Proof Theorem \ref{thm-main-ccd}] The case when $L$ is nontrivial follows by applying the case $L$ trivial to $\phi_L(S')\mid \phi_L(S)$. So it suffices to handle the case when $L$ is trivial, which we now assume.
Let $\mathscr A=A_1\bdot\ldots\bdot A_n$ be a setpartition with $\mathsf S(\mathscr A)\mid S$ and $|\mathsf S(\mathscr A)|=|S'|$ with $|X+\Sum{i=1}{n}A_i|$ maximal. In view of \cite[Proposition 10.1]{Gbook}, the hypotheses $S'\mid S$ and $n\leq |S'|\leq \mathsf h(S')$ are equivalent to such a setpartition existing. Then $\mathscr A$ is a setpartition with $\mathsf S(\mathscr A)\mid S$ maximal relative to $X$.

Suppose $|X+\Sum{i=1}{n}A_i|\geq |X|+\Sum{i=1}{n}|A_i|-n=|S'|-n+|X|$. Note we trivially have $|X+\Sigma_n(S)|\geq |X+\Sigma_n(\mathsf S(\mathscr A))|\geq |X+\Sum{i=1}{n}A_i|$. Applying Lemma \ref{lem-genn-equiazation} to $\mathscr A$ allows us to assume $\mathscr A$ is equitable (by replacing $\mathscr A$ by a modified setpartition as need be, potentially losing that $|X+\Sum{i=1}{n}A_i|$ is maximal), yielding Item 1, unless Lemma \ref{lem-genn-equiazation}.1 holds. Assume this is the case. By translating all terms of $S$ appropriately, we can w.l.o.g. assume $(A_1\cup A_2)\setminus (A_1\cap A_2)=K$ with $0\in A_1$ and $K\setminus \{0\}\subseteq A_2$. If there is some $x\in \supp(S)$ with $x\in K$, then the setpartition $\mathscr A'=A'_1\bdot A'_2$ defined by $A'_1=(A_1\setminus K)\cup \{x,x_1\}$ and $A'_2=(A_2\setminus K)\cup \{x,x_2\}$, where $x_1,x_2\in K\setminus \{x\}$ are distinct, is an equitable setpartition with $X+A'_1+A'_2=X+A_1+A_2+K$ and $|X+A'_1+A'_2|=|X+A_1+A_2|+1$. Item 1 holds in this case.
If there is some  $x\in \supp(S)$ with $x\notin K$ and $x\notin A_1\cap A_2$, then \eqref{dumpling} ensures  $H=\mathsf H(X+A_1+A_2\setminus\{y\})=\mathsf H(A_2\setminus \{y\})$ with $K\setminus\{0,y\}$ an $H$-coset, for any $y\in K\setminus\{0\}$. In such case, the setpartition  $\mathscr A'=A'_1\bdot A'_2$ defined by $A'_1=A_1\cup \{x\}$ and $A'_2=A_2\setminus \{y\}$, where $y\in  K\setminus \{0\}$, is an equitable setpartition, while Lemma \ref{lem-kt} and Kneser's Theorem imply  $|X+A'_1+A'_2|\geq |(A_1\cup \{x\})+H|+|X+(A_2\setminus\{y\})+H|-|H|\geq |X+H|+|(A_1\cup \{x\})+H|+|(A_2\setminus\{y\})+H|-2|H|=|X|+|A_1|+|A_2|-1$. Item 1 follows in this case as well. Otherwise, we have $\supp(\mathsf S(\mathscr A)^{[-1]}\bdot S)\subseteq A_1\cap A_2$, and the remaining conclusions needed for Item 3 to hold follow from Lemma \ref{lem-genn-equiazation}.1.

Next instead suppose $|X+\Sum{i=1}{n}A_i|<|X|+\Sum{i=1}{n}|A_i|-n=|S'|-n+|X|$. Let $H=\mathsf H(X+\Sum{i=1}{n}A_i)$ and $Z=\bigcap_{i=1}^n(A_i+H)$. In view of Lemma \ref{lem-modulo-seed}, we can assume $\supp(\mathsf S(\mathscr A)^{[-1]}\bdot S)\subseteq Z$ and $|(y+H)\cap A_i|\leq 1$ for all $y\in G\setminus Z$ and $i\in [1,n]$ (by replacing $\mathscr A$ by a modified setpartition as need be). This allows us to apply Lemma \ref{lem-modulo-equiazation} to add the stronger assumption that $|A_i\setminus Z|\leq 1$ for all $i$ (by replacing $\mathscr A$ by a modified setpartition as need be). But now Lemma \ref{lem-subsums=sumset} ensures that $X+\Sum{i=1}{n}A_i=X+\Sigma_n(\mathsf S(\mathscr A))=X+\Sigma_n(S)$. Thus  $H=\mathsf H(X+\Sum{i=1}{n}A_i)=\mathsf H(X+\Sigma_n(S))$, and we can   apply Lemma \ref{lem-genn-equiazation}.
Since $|\Sigma_n(S)|=|\Sigma_n(\mathsf S(\mathscr A))|=|X+\Sum{i=1}{n}A_i|<|X|+\Sum{i=1}{n}|A_i|-n=|S'|-n+|X|$, Lemma \ref{lem-modulo-equiazation}.1 cannot hold. Thus Lemma \ref{lem-modulo-equiazation}.2 allows us to further assume $\mathscr A$ is equitable  (again, by replacing $\mathscr A$ by a modified setpartition as need be),  and Item 2 follows, completing the proof.
\end{proof}

\section{Partitioning Results for Large $n$}\label{sec-nlarge}

In this section, we derive stronger results in the case our setpartition $\mathscr A=A_1\bdot\ldots\bdot A_n$ satisfies $\Sum{i=1}{n}|A_i|\leq 2n$.

\begin{lemma}\label{lem-ZisOne} Let $G$ be an abelian group, let $n\geq 1$,  let $X\subseteq G$ be a finite, nonempty subset, let $\mathscr A=A_1\bdot\ldots\bdot A_n$ be a setpartition over $G$, let $H=\mathsf H(X+\Sum{i=1}{n}A_i)$ and let $Z=\bigcap_{i=1}^n(A_i+H)$. Suppose $\Sum{i=1}{n}|A_i|\leq 2n$, \, $|A_i\setminus Z|\leq 1$ for all  $i\in [1,n]$,  and \be\nn|X+\Sum{i=1}{n}A_i|< |X+H|+\Big(\Sum{i=1}{n}|A_i|-n\Big)|H|.\ee Then  $H$ is nontrivial and $Z=\alpha+H$ for some $\alpha\in G$.
\end{lemma}

\begin{proof}
 Kneser's Theorem implies \be\label{basicbound}|X+\Sum{i=1}{n}A_i|\geq |X+H|+\Sum{i=1}{n}|A_i+H|-n|H|.\ee
If $Z=\emptyset$, then the hypothesis  $|A_i\setminus Z|\leq 1$ implies $|A_i+H|=|H||A_i|=|H|$ for all $i$. Thus
\eqref{basicbound} implies  $|X+\Sum{i=1}{n}A_i|\geq |X+H|+\Big(\Sum{i=1}{n}|A_i|-n\Big)|H|$, contrary to hypothesis. Therefore $Z$ is nonempty. If $H$ is trivial, then \eqref{basicbound} implies $|X+\Sum{i=1}{n}A_i|\geq |X|+\Sum{i=1}{n}|A_i|-n=|X+H|+\Big(\Sum{i=1}{n}|A_i|-n\Big)|H|$, contrary to hypothesis. Therefore $H$ is nontrivial. Note that $Z$ is $H$-periodic by its definition. 
If $|Z|\geq 2|H|$, then $|A_i+H|\geq 2|H|$ for all $i$, so that \eqref{basicbound} and the hypothesis $|S|\leq 2n$ imply \be\nn|X+\Sum{i=1}{n}A_i|\geq |X+H|+n|H|\geq|X+H|+\Big(\Sum{i=1}{n}|A_i|-n\Big)|H|,\ee contrary to hypothesis. Therefore $|Z|=|H|$, completing the proof.
\end{proof}

We now derive our strengthening of Theorem \ref{thm-main-ccd} for large $n$,  mirroring the main result from \cite{IttII} (which obtained the same conclusion assuming $n$ is large with respect to the exponent).

\begin{theorem}
\label{thm-partition-thm-equi} Let $G$ be an abelian group, let $n\geq 1$, let $X\subseteq G$ be a finite, nonempty subset, let $L\leq \mathsf H(X)$, let $S\in \Fc(G)$ be a sequence, and let $S'\mid S$ be a subsequence with $\mathsf h(\phi_L(S'))\leq n\leq |S'|$. Suppose $|S'|\leq 2n$. Then one of the following holds: \begin{enumerate}
\item[1.] $n=2$, $|S'|=|S|=|\supp(\phi_L(S))|$, $\supp(\phi_L(S))=\alpha+K/L$ for some $K\leq G$ and $\alpha\in G$ with $L\leq K$ and  $K/L\cong (\Z/2\Z)^2$,   $X\setminus (\beta+L)$ is $K$-periodic (or empty) for some $\beta\in X$,  and  $X+\Sigma_n(S)=X+(K\setminus L)+2\alpha$ with $|X+\Sigma_n(S)|=|X|+2|L|=(|S|-n)|L|+|X|$.
    \item[2.] There exists an \emph{equitable} setpartition $\mathscr A=A_1\bdot\ldots\bdot A_n$  with $\mathsf S(\mathscr A)\mid S$, \ $|\mathsf S(\mathscr A)|=|S'|$, $|\phi_L(A_i)|=|A_i|$ for all $i\in [1,n]$, and $|X+\Sigma_n(S)|\geq |X+\Sum{i=1}{n}A_i|\geq (|S'|-n)|L|+|X|$.
\item[3.] There exists an \emph{equitable} setpartition $\mathscr A=A_1\bdot\ldots\bdot A_n$  with $\mathsf S(\mathscr A)\mid S$, \ $|\mathsf S(\mathscr A)|=|S'|$ and $|\phi_L(A_i)|=|A_i|$ for all $i\in [1,n]$, a subgroup $K\leq H=\mathsf H(X+\Sigma_n(S))$ with $L<K$ proper, and $\alpha\in G$ such that \begin{itemize}
\item[(a)] $X+\Sigma_n(S)=X+\Sum{i=1}{n}A_i$,
\item[(b)]  $\supp(\mathsf S(\mathscr A)^{[-1]}\bdot S)\subseteq \alpha+K=\bigcap_{i=1}^n(A_i+K)$ and  $|A_i\setminus (\alpha+K)|\leq 1$ for all $i$,
\item[(c)] $|X+\Sigma_n(S)|\geq |X+H|+|S_{G\setminus (\alpha+H)}|\cdot|H|$ and
 $|X+\Sigma_n(S)|\geq |X+K|+|S_{G\setminus (\alpha+K)}|\cdot|K|$,
\item[(d)] $L+\Summ{i\in I_K}A_i=\alpha |I_K|+K$, where $I_K\subseteq [1,n]$ is the nonempty subset of all $i\in [1,n]$ with $A_i\subseteq \alpha+K$.
\end{itemize}
\end{enumerate}
\end{theorem}

\begin{proof}
As with the proof of Theorem \ref{thm-main-ccd}, it suffices to prove the case when $L$ is trivial, as we can then apply this case to $\phi_L(S')\mid \phi_L(S)$. We divide the proof into two mains cases.

\subsection*{CASE 1:}  $|X+\Sigma_n(S)|< |S'|-n+|X|$.

We will show Item 3 holds.
In this case, let $\mathcal A=A_1\bdot\ldots\bdot A_n$ be an arbitrary setpartition resulting from the application of Theorem \ref{thm-main-ccd}.2 to $S'\mid S$.
By Theorem \ref{thm-main-ccd}.2,  $\mathscr A$ is equitable, so $|A_i|\leq 2$ for all $i$ (as $|S'|\leq 2n$), \ $\mathsf S(\mathscr A)\mid S$, \ $|\mathsf S(\mathscr A)|=|S'|$, \ (a) holds, \ $\supp(\mathsf S(\mathscr A)^{[-1]}\bdot S)\subseteq Z$, and $|A_i\setminus Z|\leq 1$ for all $i$, where $Z= \bigcap_{i=1}^n(A_i+H)$ and $H=\mathsf H(X+\Sigma_n(S))$. By case hypothesis, we have $|X+\Sum{i=1}{n}A_i|=|X+\Sigma_n(S)|<|X|+\Sum{i=1}{n}|A_i|-n\leq |X|+\Big(\Sum{i=1}{n}|A_i|-n\Big)|H|$, so that
Lemma \ref{lem-ZisOne} implies $H$ is nontrivial and  $$Z=\alpha+H\quad\mbox{ for some $\alpha\in G$}.$$ But now Theorem \ref{thm-main-ccd}.2 implies  (b)  holds with $K=H$, in which case $n+|S_{G\setminus (\alpha+H)}|=\Sum{i=1}{n}|\phi_H(A_i)|$, and now (c) holds with $K=H$ by Kneser's Theorem. Since $H$ is nontrivial, the case when  $G$ is trivial is complete, allowing us to proceed by induction on $|G|\in \mathbb N\cup \{\infty\}$.

Let $I_H\subseteq [1,n]$ be all those indices $i\in [1,n]$ with $A_i\subseteq \alpha+H$, and let $I'_H\subseteq I_H$ all those indices $i\in I_H$ with $|A_i|=1$.
If $I_H=[1,n]$, then $\Sum{i=1}{n}A_i=\Summ{i\in I_H}A_i=\alpha|I_H|+H$ and (d) holds with $K=H$, yielding Item 3 with $K=H$, as desired.
Therefore, we may assume $I_H\subset [1,n]$ is a proper subset.
Since (b) and (c) hold with $K=H$, it follows that  \be\label{sugart}|X+\Sigma_n(S)|\geq |X+H|+|S_{G\setminus (\alpha+H)}|\cdot|H|=|X+H|+(n-|I_H|)|H|\geq |X|+(n-|I_H|)|H|.\ee
Since $\mathscr A$ is equitable with $|S'|\leq 2n$, we have $|A_i|\leq 2$ of all $i$, and thus  $|S'|\leq 2n-|I'_H|$. Hence the case hypothesis yields $|X+\Sigma_n(S)|\leq |S'|-n-1+|X|\leq n-1+|X|-|I'_H|$, which combines with \eqref{sugart} and $I_H\subset [1,n]$ proper to yield \be\label{icer}|I_H\setminus I'_H|\geq (n-|I_H|)(|H|-1)+1\geq |H|.\ee
Consequently,
\be\label{toobig} \Summ{i\in I_H}|A_i|-|I_H|+1=\Summ{i\in I_H\setminus I'_H}|A_i|-|I_H\setminus I'_H|+1\geq |H|+1,\ee with the final inequality following by combining \eqref{icer} with the fact that   $|A_i|=2$ for $i\in I_H\setminus I'_H$.
As a result, if $|\Summ{i\in I_H}A_i|\geq \min \{|H|,\,\Summ{i\in I_H}|A_i|-|I_H|+1\}=|H|$, then $|\Summ{i\in I_H}A_i|=|H|$ follows, in which case (d) holds with $K=H$, completing the proof as before. By translating all terms appropriately, we can w.l.o.g. assume $\alpha=0$.

Let $T=S_{H}$, let $T'=\mathscr S(\prod_{i\in I_H}^\bullet A_i)$, let $n'=|I_H|\geq |H|>0$, and let $H'=\mathsf H(\{0\}+\Sigma_{n'}(T))\leq H$. By re-indexing the $A_i$, we can w.l.o.g. assume $I_H=[1,n']$. Since $T'$ is the sequence partitioned by the setpartition $A_1\bdot\ldots\bdot A_{n'}$, it follows that $\mathsf h(T')\leq n'\leq |T'|$ (see \cite[Proposition 10.1]{Gbook}). Since the setpartition $A_1\bdot\ldots\bdot A_n$ is equitable with $|S|\leq 2n$, we have $|A_i|\in \{1,2\}$ for all $i$, so $|T'|\leq 2n'$. Since $T\in \Fc(H)$, we trivially have \be\label{important}|\Sigma_{n'}(T)|\leq |H|<\Summ{i\in I_H}|A_i|-|I_H|+1=|T'|-n'+1,\ee with the second inequality following from \eqref{toobig}.
If $H=G$, then (a) becomes  $X+\Sigma_n(S)=X+\Sum{i=1}{n}A_i=G$ with $|S_{G\setminus (\alpha+G)}|=0$, in which case (b)--(d) all follow trivially with $K=H=G$.
Therefore we may assume $H<G$ is a proper, nontrivial subgroup, and since the stabilizer $H=\mathsf H(X+\Sigma_n(S))$ of a finite set must be finite, it follows that we can apply the induction hypothesis to $\{0\}+\Sigma_{n'}(T)$ using $T'\mid T$. Then Item 3 must hold for $\Sigma_{n'}(T)$ in view of \eqref{toobig}.
Let $\mathscr B=B_1\bdot\ldots\bdot B_{n'}$ be the resulting setpartition and let  $\beta+K$, where $K\leq H'\leq H$, be the resulting coset.
Let $I_K\subseteq [1,n']=I_H$ be the subset of indices $i\in [1,n']$ with $B_i\subseteq \beta+K$. By re-indexing the $B_i$, we can w.l.o.g. assume $I_K=[1,n'']$, where $n''=|I_K|$.

Define a new setpartition $\mathscr A'=A'_1\bdot\ldots\bdot A'_n$ as follows. Set $A'_i=B_i$ for $i\in [1,n']$.
Since each $A_i$ with $i\in [1,n]\setminus I_H=[n'+1,n]$ contains a term from outside $H=\alpha+H$, it follows that $|A_i|=2$ and $|H\cap A_i|=|A_i|-1=1$ for all $i\notin I_H$.
Since $\Sum{i=1}{n'}|B_i|=\Sum{i=1}{n'}|A_i|=\Sum{i=1}{n}|A_i\cap H|$, we have $|\mathsf S(\mathscr B)^{[-1]}\bdot S_H|\geq \Sum{i=n'+1}{n}|A_i\cap H|$. This means we can take each set $A_i$ with $i\notin I_H$ and replace the element from  $A_i\cap H$  with a separate term from $\mathsf S(\mathscr B)^{[-1]}\bdot S_H$ to yield the set $A'_i$. As $|A_i\cap H|=1$ for all $i\notin I_H$,  we are guaranteed that $|A'_i|=|A_i|$ for all $i$.

By translating all terms of $S$ by $-\beta\in H$,  we can w.l.o.g. assume $\beta=0$. Since all terms of $\mathsf S(\mathscr B)^{[-1]}\bdot S_H$ are from $K=\beta+K$ by (b) (holding for $\mathscr B$), it follows that each $A'_i$, with $i>n'$, has $|A'_i\setminus K|=1$. Since (b) holds for $\mathscr B$, we also have $|A'_i\setminus K|=|B_i\setminus K|\leq 1$ for all $i\leq n'$ with $K=\bigcap_{i=1}^n(A_i+K)$. Thus, since $\Sum{i=1}{n'}B_i=\Sum{i=1}{n'}A'_i$ is $K$-periodic, Lemma \ref{lem-subsums=sumset}.1 implies  $\Sum{i=1}{n}A'_i=\Sigma_n(S)=\Sum{i=1}{n}A_i$. Hence (a) holds for $\mathscr A'=A'_1\bdot\ldots\bdot A'_n$.
 If $\Sum{i=1}{n'}B_i=\Sum{i=1}{n'}A'_i=H$, then (a)--(d) all hold for $\mathscr A'$ with $K=H$, completing the proof. Therefore we may assume $|\Sum{i=1}{n'}B_i|\leq |H|-|K|$ (as $\Sum{i=1}{n'}B_i\subseteq H$ is $K$-periodic). Thus (c) for $\mathscr B$ ensures that $|T_{H\setminus K}|\leq |H/K|-2$.
 The first part of (c) was already established. If the second fails for $\mathscr A'$, then it follows that $$|X+H|+|S_{G\setminus H}|\,|H|\leq |X+\Sigma_n(S)|< |X+K|+|S_{G\setminus K}|\,|K|\leq |X+H|+(|S_{G\setminus H}|+|H/K|-2)|K|,$$ implying $|S_{G\setminus H}|(|H/K|-1)\leq |H/H|-2$, which forces $|S_{G\setminus H}|=0$. However, in such case (a)--(d) all hold for $\mathscr A$ with $K=H$. Thus we can  assume both parts of (c) hold for $\mathscr A'$ using $K$. In view of the construction of the $A'_i$ and (b) for $\mathscr B$, it follows that (b) holds for $\mathscr A'$ with $K$, while (d) holds for $\mathscr A'$ with $K$ as it holds for $\mathscr B$. But now (a)--(d) all hold for $\mathscr A'$ with subgroup $K\leq H$, which completes CASE 1.

\subsection*{CASE 2:} $|X+\Sigma_n(S)|\geq |S'|-n+|X|$

Apply Theorem \ref{thm-main-ccd} to $S'\mid S$.
If either Theorem \ref{thm-main-ccd}.1 or Theorem \ref{thm-main-ccd}.2 holds, then  the case hypothesis ensures there exists an equitable  setpartition  $\mathcal A=A_1\bdot\ldots\bdot A_n$ with $\mathsf S(\mathscr A)\mid S$ and $|\mathsf S(\mathscr A)|=|S'|$ such that \be\label{hyp1}|X+\Sigma_n(S)|\geq |X+\Sum{i=1}{n}A_i|\geq |S'|-n+|X|.\ee It follows that Item 2 holds in this case.  Therefore we may instead assume Theorem \ref{thm-main-ccd}.3 holds, and let $\mathscr A=A_1\bdot A_2$ be the resulting setpartition. Then $|A_1|\equiv 1\mod 4$ and $|A_2|\equiv 3\mod 4$ with $|A_1|+|A_2|=|S'|\leq 2n=4$. It follows that $|A_1|=1$, $|A_2|=3$ and $A_1\cap A_2=\emptyset$. Item 1 now follows from  Theorem \ref{thm-main-ccd}.3, completing the case and proof.
\end{proof}

Finally, we conclude with the following application of Theorem \ref{thm-partition-thm-equi}, deriving some structural information regarding $S$ when, in particular, $|S|=2n$ with $|\Sigma_n(S)|\leq n+1$ and $\mathsf h(S)\leq n$.

\begin{theorem}\label{thm-special-dihedral-ample}
Let $G$ be an abelian group, let $n\geq 1$, and let $S\in \Fc(G)$ be a sequence with $|S|>n$. Suppose $|\Sigma_n(S)|\leq m+1$, where $m=\min\{n,|S|-n,|S|-\mathsf h(S)\}$. Then one of the following holds, with Items 1--4 only possible if $|\Sigma_n(S)|=m+1$ or $|\supp(S)|=1$.
\begin{itemize}
\item[1.] $n=2$, $|S|=|\supp(S)|$, and $\supp(S)=x+K$ for some $K\leq G$ and $x\in G$ with $K\cong (\Z/2\Z)^2$.
\item[2.] $m=2$ and $\supp(S)=x+K$ for some $K\leq G$ and $x\in G$ with $K\cong \Z/3\Z$.
\item[3.] $|\supp(S)|\leq 2$.
\item[4.] $\supp(S)\subseteq \{x-d,x,x+d\}$ for some $x,\,d\in G$ with $\vp_x(S)=\mathsf h(S)\geq \mathsf |S|-m$.
\item[5.] There exists $x\in G$ and  a setpartition $\mathscr A=A_1\bdot\ldots\bdot A_n$ with
 $\mathsf S(\mathscr A)\mid S$, $|\mathsf S(\mathscr A)|=n+m$,
$\Sum{i=1}{n}A_i=\Sigma_n(S)$,  $\supp(\mathsf S(\mathscr A)^{[-1]}\bdot S)\subseteq x+H$, $|A_i|\leq 2$ and $(x+H)\cap A_i\neq \emptyset$ for all $i\in [1,n]$, and $|\Sum{i=1}{n}A_i|=|\underset{i\neq j}{\Sum{i=1}{n}}A_i|$ for some $j\in [1,n]$, where  $H=\mathsf H(\Sigma_n(S))$ is nontrivial.
\end{itemize}
\end{theorem}

\begin{proof} If $\mathsf h(S)=|S|$, then $|\supp(S)|=1$, and Item 3 holds. Therefore we may assume $\mathsf h(S)<|S|$, and so, since $|S|>n$, may let $m$ be the maximal integer in $[1,n]$ such that there is a subsequence $S'\mid S$ with $|S'|=n+m$ and $\mathsf h(S')\leq n$. If $m=n$, then $2n\leq |S|$ and $\mathsf h(S)\leq |S|-n$, whence $m=n=\min\{n,|S|-n,|S|-\mathsf h(S)\}\geq 1$. If $S'=S$, then $|S|=|S'|=n+m\leq 2n$ and $\mathsf h(S)\leq |S|-m=n$, whence  $m=|S|-n=\min\{n,|S|-n,|S|-\mathsf h(S)\}\geq 1$. If $m<n$ and $S'$ is a proper subsequence, then the maximality of $m$ ensures  that $(S')^{[-1]}\bdot S$ has only one distinct term, say $x$. Now $\vp_x(S')\leq n$. If $\vp_x(S')<n$, then $S''=S'\bdot  x$ is a subsequence with $|S''|=|S'|+1=n+m+1$, $\mathsf h(S'')\leq n$ and $m+1\leq n$, so $m+1$ contradicts the maximality of $m$. Therefore $\mathsf h(S)=\vp_x(S)=|S|-|S'|+n=|S|-m\geq |S|-n$ in this case, implying $\mathsf h(S)=|S|-|S'|+n\geq n$ and   $m=|S|-\mathsf h(S)=\min\{n,|S|-n,|S|-\mathsf h(S)\}\geq 1$. In consequence, in all possible cases, we deduce that \be\label{m-def}m=\min\{n,|S|-n,|S|-\mathsf h(S)\}\geq 1.\ee Let $H=\mathsf H(\Sigma_n(S))$.

By hypothesis, $|\Sigma_n(S)|\leq m+1=|S'|-n+1$. If $|\Sigma_n(S)|\leq m$, then Theorem \ref{thm-main-ccd}.2 applied to $S'\mid S$ (with $X=\{0\}$) yields a setpartition $\mathscr A=A_1\bdot\ldots\bdot A_n$ with $\mathsf S(\mathscr A)\mid S$, $|\mathsf S(\mathscr A)|=|S'|=n+m$,  $\Sum{i=1}{n}A_i=\Sigma_n(S)$,  $\supp(\mathsf S(\mathscr A)^{[-1]}\bdot S)\subseteq Z$, and $|A_i|\leq 2$ and $|A_i\setminus Z|\leq 1$ for all $i$, where $Z=\bigcap_{i=1}^n(A_i+H)$. By indexing the $A_i$ appropriately, we can assume $|A_i|=2$ for $i\in [1,m]$. By Lemma \ref{lem-ZisOne}, $H$ is nontrivial and $Z=x+H$ for some $x\in G$.
If $|\Sum{i=1}{j}A_i|>|\Sum{i=1}{j-1}A_i|$ for all $j\in[2,m]$, then it follows that $|\Sigma_n(S)|=|\Sum{i=1}{n}A_i|\geq m+1$, contrary to assumption. Thus there is some $j\in [2,m]$ with $|\Sum{i=1}{j}A_i|=|\Sum{i=1}{j-1}A_i|$, meaning Item 5 holds. It remains to consider the case when $|\Sigma_n(S)|=m+1=|S'|-n+1=\Sum{i=1}{n}|A_i|-n+1$.

Suppose $m=1$. If $1=m=|S|-\mathsf h(S)$, then $\mathsf h(S)=|S|-1$, implying $|\supp(S)|\leq 2$, so Item 3 follows. If $1=m=n$, then  $|\supp(S)|=|\Sigma_1(S)|=|\Sigma_n(S)|\leq m+1=2$, and Item 3 follows. If $1=m=|S|-n$, then $n=|S|-1$ and $|\supp(S)|=|\Sigma_1(S)|=|\sigma(S)-\Sigma_{|S|-1}(S)|=|\Sigma_n(S)|\leq m+1=2$, and Item 3 again follows. So we may now assume $m\geq 2$.

Apply Theorem \ref{thm-partition-thm-equi} (with $X=\{0\}$) to $\Sigma_{n}(S)$ with $S'\mid S$. If Theorem \ref{thm-partition-thm-equi}.1 holds, then Item 1 follows. Otherwise, in view of $|\Sigma_n(S)|\leq m+1=|S'|-n+1$,  let $\mathscr A=A_1\bdot\ldots\bdot A_n$ be the resulting \emph{equitable} setpartition with   $Z=\bigcap_{i=1}^n(A_i+H)$, \be\label{monka}\mathsf S(\mathscr A)\mid S, \quad |\mathsf S(\mathscr A)|=|S'|=n+m, \quad  \Sum{i=1}{n}A_i=\Sigma_n(S) \quad\und\quad  |A_i|=2\;\mbox{ for all $i\in [1,m]$}.\ee

\subsection*{CASE 1} For any setpartition $\mathscr A$ satisfying \eqref{monka}, we have $|\Sum{i=1}{j}A_i|\geq |\Sum{i=1}{j-1}A_i|+1$ for all $j\in [2,m]$.

\smallskip

In this case, Lemma \ref{lem-local-cdtbound} implies $|\Sum{i=1}{n}A_i|\geq m+1$, with equality only possible if equality holds in each estimate $|\Sum{i=1}{j}A_i|\geq |\Sum{i=1}{j-1}A_i|+1$ for $j\in [2,m]$. As this is the case, $|\Sum{i=1}{j}A_i|=|\Sum{i=1}{j-1}A_i|+1$ for all $j\in [2,m]$. Moreover, this must be true under any re-indexing of the $A_i$ with $i\in [1,m]$, whence  each $A_i$ is an arithmetic progression with a common difference $d\in G$, and  each $\Sum{i=1}{j}A_i$ is also an arithmetic progression with difference $d$ and length $j+1$ for $j\in [1,m]$. In particular, $$3\leq m+1\leq \ord(d),$$  and $\Sigma_n(S)=\Sum{i=1}{n}A_i$ is an arithmetic progression with difference $d$, whence either $H$ is trivial or $H=\la d\ra$.  Thus $\underset{i\neq j}{\Sum{i=1}{n}}A_i$ is aperiodic for any $j\in [1,m]$.
Moreover, if $H=\la d\ra$, then $m=\ord(d)-1$ and $\Sigma_n(S)=\Sum{i=1}{n}A_i$ is a single $H$-coset, which in view of $|S|>n$ is only possible if $\supp(S)$ is contained in a single $H$-coset.

 Suppose some pair $A_i$ and $A_j$ are disjoint with $i,\,j\in [1,m]$, say $A_{m}=\{x,x+d\}$ and $A_{m-1}=\{y,y+d\}$.
 Then $y\notin\{x+d,x,x-d\}$ and $$\Sigma_2(x\bdot (x+d)\bdot y\bdot (y+d))=\{x+y,x+y+d,x+y+2d, 2y+d,2x+d\}$$ is a set of cardinality at least $4$.
 Thus, since $A_1+\ldots+A_{m-2}+\Sigma_2(x\bdot (x+d)\bdot y\bdot (y+d))+A_{m+1}+\ldots+A_n\subseteq \Sigma_n(S)$ with
 $|\Sigma_n(S)|=m+1$, we must have $m\geq 3$.
Now $\Sum{i=1}{m-2}A_i$ is an arithmetic progression with difference $d$ and length $2\leq m-1\leq \ord(d)-2$, but $\{x,y\}$ is \emph{not} an arithmetic progression with difference $d$ since $y\notin \{x+d,x,x-d\}$. It follows that $|\Sum{i=1}{m-2}A_i+\{x,y\}|\geq m+1$.
Thus, since $|\Sigma_n(S)|=m+1$, we conclude that \be\label{tabletap}|\Sigma_n(S)|=|\Sum{i=1}{m-2}A_i+\{x,y\}|=|\Sum{i=1}{m-2}A_i+\{x,y\}+\{x+d,y+d\}|,\ee implying that $\Sum{i=1}{m-2}A_i+\{x,y\}$ is a translate of $\Sigma_n(S)=\Sum{i=1}{m-2}A_i+\{x,y\}+\{x+d,y+d\}+A_{m+1}+\ldots+A_n$  with  $x-y\in H= \mathsf H(\Sum{i=1}{m-2}A_i+\{x,y\})$.
In such case, $H$ is nontrivial as $x\neq y$, so we have $m=\ord(d)-1$ and $\supp(S)\subseteq x+H=x+\la d\ra$ by the observation at the end of the previous paragraph.
Letting $\mathscr A'= A'_1\bdot\ldots\bdot A'_n$, where $A'_{m-1}=\{x,y\}$, $A'_m=\{x+d,y+d\}$ and $A'_i=A_i$ for $i\neq m-1,m$, it follows in view of \eqref{tabletap} that Item 5 holds. So we can now assume $A_i\cap A_j\neq \emptyset$ for all $i,j\in [1,m]$.
Thus, since each $A_i$ is an arithmetic progression with difference $d$, it follows that there must be some $x\in \bigcap_{i=1}^mA_i$ (this is trivially true if  $\ord(d)=3$, as then $m\leq \ord(d)-1=2$). Thus $A_1\cup \ldots\cup A_m\subseteq \{x-d,x,x+d\}$ with $x\in A_i$ for all $i\in [1,m]$.

If there is some $y\in \supp(S)\setminus \{x-d,x,x+d\}$, then we can exchange the term equal to $x\pm d$ in $A_m$ with $y$, resulting in a set $A'_m=\{x,y\}$ that is not an arithmetic progression with difference $d$, while $A_1$ remains an arithmetic progression with difference $d$ as $m\geq 2$. Since $\underset{i\neq m}{\Sum{i=1}{n}}A_i$ is aperiodic, Knseser's Theorem ensures the resulting setpartition (replacing $A_m$ by $A'_m$) satisfies \eqref{monka}, and so repeating the above arguments using the setpartition $A_1\bdot\ldots A_{m-1}\bdot A'_m\bdot A_{m+1}\bdot\ldots\bdot A_n$ completes the proof. Therefore we may instead assume $\supp(S)\subseteq \{x-d,x,x+d\}$. Indeed, we may assume $\supp(S)=\{x-d,x,x+d\}$, else Item 3 holds.
If $\ord(d)=3$, then $2\leq m\leq \ord(d)-1$ forces $m=2$. In this case,  $\supp(S)=x+K$ with $K=\{0,d,-d\}$ a subgroup of size $3$, and Item 2 follows. Therefore we can assume $\ord(d)\geq 4$.

Suppose there is a term $y\in \supp((A_1\bdot\ldots\bdot A_m)^{[-1]}\bdot S)$ with $y\neq x$.
Since $\supp(S)=\{x-d,x,x+d\}$, we have $y=x\pm d$, say w.l.o.g. $y=x+d$. If $A_i=\{x,x+d\}$ for all $i\in [1,m]$, then either $|\supp(S)|=2$, yielding Item 3, or else we can exchange $y$ for some $y'=x-d\in \supp((A_1\bdot\ldots\bdot A_m)^{[-1]}\bdot S)$. Thus, swapping $y$ as need be, we obtain that  there is some $A_i$ with $i\in [1,m]$, say $A_m$, with $y\notin A_m$.
 Then w.l.o.g. $y=x+d$ and $A_m=\{0,x-d\}$. Note we either have $y\in \supp(\mathsf S(\mathscr A)^{[-1]}\bdot S)$ or $A_k=\{y\}=\{x+d\}$ for some $k>m$.
 Define a new setpartition  $\mathscr A'=A'_1\bdot\ldots\bdot A'_n$ with  $A'_i=A_i$ for $i\leq m$,  $A'_m=\{x-d,x+d\}$, and either $A'_i=A_i$ for all $i>m$ (if $y\in \supp(\mathsf S(\mathscr A)^{[-1]}\bdot S)$) or else $A'_k=\{x\}$ and $A'_i=A_i$ for all $i\in [m+1,n]\setminus \{k\}$ (if $y\notin \supp(\mathsf S(\mathscr A)^{[-1]}\bdot S)$).
Since $\ord(d)\geq 4$, it follows that $A'_m$ is not an arithmetic progression with difference $d$, while each $A'_i$ with $i\in [1,m-1]$ is. Since $m\geq 2$, repeating the above arguments using the setpartition $\mathscr A'$ completes the proof (as $\underset{i\neq m}{\Sum{i=1}{n}}A_i$ is aperiodic, Kneser's Theorem ensures \eqref{monka} holds for $\mathscr A'$). So we instead assume $\supp((A_1\bdot\ldots\bdot A_m)^{[-1]}\bdot S)\subseteq \{x\}$. Combined with $x\in A_i$ for all $i\in [1,m]$, we find $\vp_x(S)=\mathsf h(S)\geq |S|-m$, and now Item 4 holds, completing CASE 1.

\subsection*{CASE 2} There is some setpartition $\mathscr A$ satisfying \eqref{monka} with $|\Sum{i=1}{n}A_i|=|\underset{i\neq j}{\Sum{i=1}{n}}A_i|$ for some $j\in [2,m]$.

\smallskip

By  case hypothesis and Kneser's Theorem, $H=\mathsf H(\Sigma_n(S))$ is nontrivial.
If $m+1=|\Sigma_n(S)|=|H|$, then $\Sigma_n(S)$ is an $H$-coset, which in view of $|S|>n$ is only possible if $\supp(S)\subseteq x+H$ for some $x\in G$. Hence Item 5 holds in view of the case hypothesis.  So we now assume $|\Sum{i=1}{n}A_i|=|\Sigma_n(S)|\geq 2|H|$.
Thus $\Sum{i=1}{n}|\phi_H(A_i)|\geq n+1$.
Let $\mathscr B=B_1\bdot\ldots\bdot B_n$ be a setpartition with $\mathsf S(\mathscr B)\mid S$, $|\mathsf S(\mathscr B)|=n+m$, and $\Sum{i=1}{n}B_i=\Sigma_n(S)$ such that, letting $I_2\subseteq [1,n]$ be the subset of all $i\in [1,n]$ with $|\phi_H(B_i)|\geq 2$, the following hold
 \begin{itemize}
 \item[M1.] For each $i\in I_2$, there is some $b_i\in B_i$ such that   $|\phi_H\big(B_i\setminus (b_i+H)\big)|=|B_i\setminus (b_i+H)|$, and
 \item[M2.] either $\Summ{i\in I_2}|\phi_H(B_i)|>2|I_2|$ or $|\phi_H(B_{i'})|=|B_{i'}|$ for some $i'\in I_2$
 \end{itemize}
Since $|A_i|\leq 2$ for all $i$ and $\Sum{i=1}{n}|\phi_H(A_i)|\geq n+1$, \ $\mathscr A$ satisfies all these hypotheses. Let $I_1=[1,n]\setminus I_2$ be the subset of all $i\in [1,n]$ with $|\phi_H(B_i)|=1$, and re-index the $B_i$ so that $I_1=[1,|I_1|]$. Kneser's Theorem implies $|\Sum{i=1}{n}B_i|\geq \Summ{i\in I_2}|B_i+H|-(|I_2|-1)|H|\geq \Summ{i\in I_2}|B_i|+(|I_2|+1)(|H|-1)-(|I_2|-1)|H|=\Summ{i\in I_2}|B_i|-|I_2|+(2|H|-1)$, with the latter inequality in view of conditions M1 and M2 (note $\Summ{i\in I_2}|\phi_H(B_i)|\geq 2|I_2|$ holds trivially in view of $|\phi_H(B_i)|\geq 2$ for $i\in I_2$). Combined with the inequality $|\Sum{i=1}{n}B_i|=|\Sigma_n(S)|\leq |S'|-n+1=\Sum{i=1}{n}|B_i|-n+1=\Sum{i=1}{n}|B_i|-|I_1|-|I_2|+1$, we find  \be\label{madrag}\Summ{i\in I_1}|B_i|\geq |I_1|+2|H|-2.\ee Consequently, $I_1$ is nonempty, and since we trivially have $|\Summ{i\in I_1}B_i|\leq |H|$ (as each $B_i$ with $i\in I_1$ is contained in an $H$-coset), it follows that $|\Summ{i\in I_1}B_i|\leq |H|\leq \Summ{i\in I_1}|B_i|-|I_1|-(|H|-2)< \Summ{i\in I_1}|B_i|-|I_1|+1$,
 with the later inequality holding since $H$ is nontrivial (as noted at the start of the case).
 Lemma \ref{lem-local-cdtbound} now implies  there is some $j\in [2,|I_1|]$ with $|\Sum{i=1}{j}B_i|<|\Sum{i=1}{j-1}B_i|+|B_j|-1$, in which case Theorem \ref{thm-multbound} implies \be\label{pullout}\Sum{i=1}{j-1}B_i+(B_j\setminus \{y\})=\Sum{i=1}{j}B_i\quad\mbox{ for all $y\in B_j$}.\ee In particular, $|B_j|\geq 2$. Also, since $|B_i|\geq 2$ for all $i\in I_2$, and since $\Sum{i=1}{n}|B_i|=|S'|\leq 2n$, it follows that \be\label{smallside} \Summ{i\in I_1}|B_i|\leq 2|I_1|\quad \und\quad |I_1|\geq 2|H|-2\geq |H|,\ee with the latter inequality above following from the former combined with \eqref{madrag}.

Now additionally assume that our setpartition $\mathscr B$ is chosen, subject to  $\mathsf S(\mathscr B)\mid S$, $|\mathsf S(\mathscr B)|=n+m$,  $\Sum{i=1}{n}B_i=\Sigma_n(S)$, M1 and M2, so that
\begin{itemize}\item[M3.] $\Sum{i=1}{n}|\phi_H(B_i)|$ is maximal.\end{itemize}
Since $\mathscr A$ satisfies the defining conditions for $\mathscr B$, we have $\Sum{i=1}{n}|\phi_H(B_i)|\geq \Sum{i=1}{n}|\phi_H(A_i)|\geq n+1$, ensuring that $I_2$ is nonempty. We claim that this ensures $B_j+H\subseteq B_i$ for all $i\in [1,n]$, where $j\in I_1$ is the index defined above. Indeed, if this fails, then there is some $x\in B_j$ and $k\in [1,n]$ with $\phi_H(x)\notin \phi_H(B_k)$.
In this case, remove $x$ from $B_j$ and place it in $B_k$ to yield a new setpartition $\mathscr B=B'_1\bdot\ldots\bdot B'_n$, where $B'_j=B_j\setminus \{x\}$, $B'_k=B_k\cup \{x\}$ and $B'_i=B_i$ for $i\neq j,k$. In view of \eqref{pullout}, we have $\mathsf S(\mathscr B')=\mathsf S(\mathscr B)$, $|\mathsf S(\mathscr B')|=|\mathsf S(\mathscr B)|=n+m$ and $\Sum{i=1}{n}B'_i=\Sigma_n(S)$. Since
$\phi_H(x)\notin \phi_H(B_k)$, it follows that $x$ is the unique element from its $H$-coset in $B'_k$, so M1 and M2 also hold for $\mathcal B'$. However, since $|\phi_H(B_j)|=|\phi_H(B'_j)|=1$ and $|\phi_H(B'_k)|=|\phi_H(B_k)|+1$, we see that $\mathscr B'$ contradicts the maximality of $\Sum{i=1}{n}|\phi_H(B_i)|$ for $\mathscr B$ given in M3. Therefore, $B_j+H\subseteq B_i$ for all $i$, as claimed. Letting $x\in B_j$ and recalling that $B_j$ is contained in an $H$-coset (as $j\in I_1$), it follows that $x+H=\bigcap_{i=1}^n(B_i+H)$. Likewise, if there were some $y\in \supp(\mathsf S(\mathscr B)^{[-1]}\bdot S)$ with $\phi_H(y)\neq \phi_H(x)$, then  we could remove $x$ from $B_j$ and place $y$ in $B_j$ to yield a new setpartition $\mathscr B=B'_1\bdot\ldots\bdot B'_n$, where $B'_j=B_j\setminus \{x\}\cup \{y\}$ and  $B'_i=B_i$ for $i\neq j$, which would again contradict the maximality of $\mathscr B$ given in M3. Therefore we may assume otherwise. In summary,
\be \supp(\mathsf S(\mathscr B)^{[-1]}\bdot S)\subseteq x+H=\bigcap_{i=1}^n(B_i+H).\label{maotau}\ee

\subsection*{Claim A} $(y+H)\cap B_i=\{y\}$ for any $i\in [1,n]$ and $y\in B_i\setminus (x+H)$.

\begin{proof}

Assume by contradiction there is some $k\in [1,n]$  and $y\in B_k\setminus (x+H)$ with $|(y+H)\cap B_k|=r\geq 2$.  Since $B_i\subseteq x+H$ for each $i\in I_1$ by \eqref{maotau}, we must have $k\in I_2$.
Let $\mathscr C=C_1\bdot\ldots\bdot C_n$ be a setpartition with $\mathsf S(\mathcal C)=\mathsf S(\mathcal B)$ and  $\Sum{i=1}{n}B_i=\Sigma_n(S)$ such that
$C_i=B_i$ for all $i\in I_2\setminus \{k\}$, \
$C_k\setminus B_k\subseteq x+H$, \ $C_k\cap B_k=B_k\setminus \{y_1,\ldots,y_t\}$, $C_{|I_1|+1-i}\setminus (x+H)=\{y_{i}\}\subset C_{|I_1|+1-i}$ for $i\in [1,t]$,  where $y_1,\ldots,y_t\in (y+H)\cap B_k$ are $t\in [0,r-1]$ distinct elements, and
(subject to these conditions) $|(x+H)\cap C_k|$ is maximal, and then (subject to prior conditions) $t\geq 0$ is maximal.  Note $\mathscr B$ satisfies these conditions with $t=0$, so $\mathscr C$ exists. The defining conditions for $\mathscr C$ ensure $$I'_2=I_2\cup [|I_1|+1-t,|I_1|]$$ is the subset of indices $i\in [1,n]$ with $|\phi_H(C_i)|\geq 2$ and that M1 holds for all $C_i$ with $i\in I'_2\setminus \{k\}$.

Suppose $t=r-1$. The defining conditions for $\mathscr C$ along with M1 for $\mathscr B$ ensure  all elements from $C_k\setminus (\{x,y\}+H)$ are the unique element from their $H$-coset in $C_k$ with $|(y+H)\cap C_k|=r-t$. Thus, since $t=r-1$, we see that M1 holds for $\mathscr C$.  The defining conditions for $\mathscr C$  ensure $\phi_H(C_i)=\phi_H(B_i)$ for $i\in I_2$ and $i\in I_1\setminus [|I_1|+1-t,|I_1|]$, while $\phi_H(B_{|I_1|+1-i})\subset \phi_H(C_{|I_1|+1-i})$ for $i\in [1,t]$; moreover, $C_i=B_i$ for $i\in I_2\setminus \{k\}$. Thus, since M2 holds for $\mathscr B$, it also holds for $\mathscr C$ (note $i'\neq k$ in M2 as $|(y+H)\cap B_k|\geq 2$), and $\Sum{i=1}{n}|\phi_H(C_i)|=\Sum{i=1}{n}|\phi_H(B_i)|+t=\Sum{i=1}{n}|\phi_H(B_i)|+r-1>\Sum{i=1}{n}|\phi_H(B_i)|$. Hence $\mathscr C$ contradicts the maximality condition M3 for $\mathscr B$. So we instead assume $t<r-1$, meaning $|(y+H)\cap C_k|=r-t\geq 2$.

Suppose $(x+H)\cap C_k=x+H$. Then $C_i\subseteq x+H\subseteq  C_k$ for all $i\in [1,|I_1|-t]$ (since $\phi_H(C_{i})=\phi_H(B_i)$ for all $i\notin [|I_1|+1-t,|I_1|]$). Let $y_{t+1}\in (y+H)\cap C_k$.
In view of \eqref{smallside}, we have $|I_1|\geq |H|\geq r\geq t+2$, so we can define a new setpartition $\mathcal C'=C'_1\bdot\ldots\bdot C'_n$, where $C'_k=C_k\setminus \{y_{t+1}\}$, $C'_{|I_1|-t}=C_{|I_1|-t}\cup \{y_{t+1}\}$, and $C'_i=C_i$ for all $i\neq k,|I_1|-t$. Then $\mathsf S(\mathscr C')=\mathsf S(\mathscr C)$. We have $C_{|I_1|-t}\subseteq x+H=(x+H)\cap C_k$ and $y_{t+1}\in y+H\neq x+H$. Thus $C_{|I_1|-t}\subseteq C_k\setminus \{y_{t+1}\}$ and  $C_{|I_1|-t}+C_k\subseteq (C_{|I_1|-t}\cup  \{y_{t+1}\})+(C_k\setminus \{y_{t+1}\})$, ensuring $\Sigma_n(S)=\Sum{i=1}{n}C_i\subseteq \Sum{i=1}{n}C'_i\subseteq \Sigma_n(S)$, forcing equality to hold. But now, since $t+1\leq r-1$, we see that $\mathscr C'$ contradicts the maximality of $t$ for $\mathscr C$.
So we instead conclude that
\be\label{propertea} (x+H)\cap C_k\subset x+H.\ee

Note $\rho:=|H|-|(x+H)\cap C_k|\geq 1$ by \eqref{propertea}.
Since $C_k\setminus B_k\subseteq x+H$ and $C_k\cap B_k=B_k\setminus \{y_1,\ldots,y_t\}$ with $t<r$, it follows from M1 for $\mathscr B$ that \be\label{talc}|(C_k+H)\setminus C_k|\geq  (|\phi_H(C_k)|-2)(|H|-1)+t+\rho\geq t+1.\ee
We also have
$|(C_i+H)\setminus C_i|\geq (|\phi_H(C_i)|-1)(|H|-1)\geq |H|-1$ for $i\in I_2\setminus \{k\}$ (by M1 for $\mathscr B$), \ either $|\phi_H(C_k)|\geq 3$ (improving the final estimate in \eqref{talc} by $|H|-1$) or $|(C_i+H)\setminus C_i|\geq 2(|H|-1)$ for some $i\in I_2\setminus \{k\}$ (by M2 for $\mathscr B$, noting that $i'\neq k$ in view of $|(y+H)\cap B_k|\geq 2$), and $|(C_i+H)\setminus C_i|=|H|-1$ for $i\in [|I_1|+1-t,|I_1|]$. As a result,
\begin{align*}\Summ{i\in I'_2}|(C_i+H)\setminus C_i|&\geq\Summ{i\in I'_2}|C_i|+|I'_2|(|H|-1)+t+1.\end{align*}
Combining this estimate with Kneser's Theorem, we obtain   $|\Sum{i=1}{n}C_i|\geq \Summ{i\in I'_2}|C_i+H|-(|I'_2|-1)|H|\geq \Summ{i\in I'_2}|C_i|-|I'_2|+t+1+|H|$. Combined with the inequality $|\Sum{i=1}{n}C_i|=|\Sigma_n(S)|\leq |S'|-n+1=\Sum{i=1}{n}|C_i|-n+1=\Sum{i=1}{n}|C_i|-|I'_1|-|I'_2|+1$, where $I'_1:=[1,n]\setminus I'_2=[1,|I_1|-t]$, we find  \be\label{madrag2}\Summ{i\in I'_1}|C_i|\geq |I'_1|+|H|+t\geq |I'_1|+|H|.\ee Consequently, since we trivially have $|\Summ{i\in I'_1}C_i|\leq |H|$ (as each $C_i$ with $i\in I'_1$ is contained in an $H$-coset), it follows that $|\Summ{i\in I'_1}C_i|\leq |H|< \Summ{i\in I'_1}|C_i|-|I'_1|+1$. As before, this ensures via Lemma \ref{lem-local-cdtbound} and Theorem \ref{thm-multbound} that there is some $j'\in [2,|I'_1|]$ with
\be\label{removal} \Sum{i=1}{j'-1}C_i+(C_{j'}\setminus \{z\})=\Sum{i=1}{j}C_i\quad \mbox{ for all $z\in C_{j'}$}.\ee

Suppose $C_{j'}\subseteq C_k$. Let $y_{t+1}\in (y+H)\cap C_k$.
In view of \eqref{smallside}, we have $|I_1|\geq |H|\geq r\geq t+2$, so we can define a new setpartition $\mathcal C'=C'_1\bdot\ldots\bdot C'_n$, where $C'_k=C_k\setminus \{y_{t+1}\}$, $C'_{j'}=C_{j'}\cup \{y_{t+1}\}$, and $C'_i=C_i$ for all $i\neq k,j'$. Then $\mathsf S(\mathscr C')=\mathsf S(\mathscr C)$. We have $C_{j'}\subseteq C_k$, and thus $C_{j'}\subseteq C_k\setminus \{y_{t+1}\}$ (since $j'\in I'_1$ ensures $C_{j'}\subseteq x+H\neq y+H$ and $y_{t+1}\in y+H$). Hence $C_{j'}+C_k\subseteq (C_{j'}\cup  \{y_{t+1}\})+(C_k\setminus \{y_{t+1}\})$, ensuring $\Sigma_n(S)=\Sum{i=1}{n}C_i\subseteq \Sum{i=1}{n}C'_i\subseteq \Sigma_n(S)$, forcing equality to hold. But now, since $t+1\leq r-1$, we see that $\mathscr C'$ contradicts the maximality of $t$ for $\mathscr C$ (re-indexing the $C'_i$ with $i\in I'_1$ so that $j'=|I'_1|$). So we instead conclude that $C_{j'}\nsubseteq C_k$.

Since $C_{j'}\nsubseteq C_k$ and $C_{j'}\subseteq x+H$ (as $j'\in I'_1$),  there is some $z\in C_{j'}\setminus C_k$ with $z\in x+H$. Define a new setpartition $\mathcal C'=C'_1\bdot\ldots\bdot C'_n$, where $C'_{j'}=C_{j'}\setminus \{z\}$, $C'_{k}=C_{k}\cup \{z\}$, and $C'_i=C_i$ for all $i\neq j',k$. Then $\mathsf S(\mathscr C')=\mathsf S(\mathscr C)$, and \eqref{removal} ensures $\Sigma_n(S)=\Sum{i=1}{n}C_i\subseteq \Sum{i=1}{n}C'_i\subseteq \Sigma_n(S)$, in which case equality holds. But now $|(x+H)\cap C'_k|=|(x+H)\cap C_k|+1$, so that $\mathscr C'$ contradicts the maximality of $|(x+H)\cap C_k|$ for $\mathscr C$, completing Claim A
\end{proof}

Since $H$ is nontrivial (as noted at the start of CASE 2) and  $m+1=|\Sigma_n(S)|\leq |S'|-n+1=\Sum{i=1}{n}|B_i|-n+1$, Claim A allows us to  apply Lemma \ref{lem-modulo-equiazation} to $\mathscr B$ (with $X=\{0\}$), giving the existence of  a setpartition $\mathscr C=C_1\bdot\ldots\bdot C_\ell$ with $\mathsf S(\mathscr C)=\mathsf S(\mathscr B)$,  $\Sum{i=1}{n}C_i=\Sum{i=1}{n}B_i=\Sigma_n(S)$, $(x+H)\subseteq Z=\bigcap_{i=1}^n(C_i+H)$, and  $|C_i\setminus Z|\leq 1$ for all $i$.
If $Z\neq x+H$, then $m=n$ and  $|\phi_H(C_i)|=2$ for all $i$ (recall $|S'|=n+m\leq 2n$), whence Kneser's Theorem implies $|S'|-n+1=|\Sigma_n(S)|=|\Sum{i=1}{n}C_i|\geq (n+1)|H|\geq (|S'|-n+1)|H|$, contradicting that $H$ is nontrivial. Therefore $Z=x+H$.
It necessarily follows that $\Sum{i=1}{n}|\phi_H(C_i)|=\Sum{i=1}{n}|\phi_H(B_i)|$ since $x+H=\bigcap_{i=1}^n(B_i+H)= \bigcap_{i=1}^n(C_i+H)$ with $|(y+H)\cap C_i|\leq 1$ and $|(y+H)\cap B_i|\leq 1$ for all $i\in [1,n]$ and $y+H\neq x+H$ (cf. Claim A and Lemma \ref{lem-modulo-equiazation}).  Applying Lemma \ref{lem-genn-equiazation} (with $X=\{0\}$) allows us to replace $\mathscr C$ with a setpartition having all the defining properties for $\mathscr C$ and which is  equitable (Lemma \ref{lem-genn-equiazation}.1 cannot hold since $H=\mathsf H(\Sigma_n(S))$ is nontrivial), so we gain that $|C_i|\leq 2$ for all $i$. In doing so, we find that $\mathscr C$ now satisfies the defining conditions for $\mathscr B$. Thus we can w.l.o.g. assume the setpartition $\mathscr B$ defined above has $|B_i|\leq 2$ for all $i$.
 In view of \eqref{madrag}, there are at least  $2|H|-2\geq |H|$ sets $B_i$ with $|B_i|=2$ and $i\in I_1$. Since $|\Summ{i\in I_1}B_i|\leq |H|$ in view of each $B_i$ being contained in an $H$-coset for $i\in I_1$, it now follows by a simple greedy algorithm \cite[Proposition 2.2]{hypergraph-egz} that there is a subset $J_1\subset I_1$ with $|J_1|\leq |H|-1$ and $|\Summ{i\in J_1}B_i|=|\Summ{i\in I_1}B_i|$. Recalling \eqref{maotau}, we find Item 5 holds using the setpartition  $\mathscr B$, completing the case and  proof.
\end{proof}

\end{document}